\documentclass[12pt,a4paper]{amsart}

\usepackage{amsmath,amsfonts,amssymb,mathtools,extarrows}

\usepackage[hmargin=2.5cm,vmargin=2.5cm]{geometry}
\usepackage[dvipsnames]{xcolor}	
\usepackage{hyperref}
\usepackage{nameref,zref-xr}     

\usepackage{mathbbol}
\usepackage{cjhebrew}
\usepackage{pdftexcmds}

\usepackage{tikz}
\usepackage{float}
\usetikzlibrary{decorations.pathmorphing}

\hypersetup{
colorlinks=true, linktocpage=true, pdfstartpage=1, pdfstartview=FitV,breaklinks=true, pdfpagemode=UseNone, pageanchor=true, pdfpagemode=UseOutlines,plainpages=false, bookmarksnumbered, bookmarksopen=true, bookmarksopenlevel=1,hypertexnames=true, pdfhighlight=/O,urlcolor=webbrown, linkcolor=RoyalBlue, citecolor=ForestGreen}

\usepackage{enumerate}      

\usepackage{scalerel}         
\usepackage{stmaryrd}

\usepackage[all]{xy}

\usepackage[textsize=footnotesize,textwidth=0.85in]{todonotes}
\presetkeys%
    {todonotes}%
    {color=Apricot}{}%
\newcommand{\oM}{\overline{\mathcal M}}

\def\oM{{\overline{\mathcal{M}}}}

\def\d{{\partial}}

\newcommand{\eps}{\varepsilon}

\newcommand{\DR}{\mathrm{DR}}

\newcommand{\Coef}{\mathrm{Coef}}

\newcommand{\SRT}{\mathrm{SRT}}
\newcommand{\PSSRT}{\mathrm{PSSRT}}

\newcommand{\LDLSRT}{\mathrm{LDLSRT}}

\newcommand{\ZZ}{\mathbb{Z}}
\newcommand{\QQ}{\mathbb{Q}}

\newcommand{\gl}{\mathrm{gl}}

\newcommand{\hhh}{\text{{\Large{\<h>}}}}

\newtheorem{theorem}{Theorem}[section]

\newtheorem{lemma}[theorem]{Lemma}
\newtheorem{corollary}[theorem]{Corollary}
\newtheorem{conjecture}[theorem]{Conjecture}

\theoremstyle{remark}

\newtheorem{remark}[theorem]{Remark}
\newtheorem{example}[theorem]{Example}

\theoremstyle{definition}

\usepackage{color}

\newcommand{\bUp}{\Xi\Xi}

\newcommand{\bOmT}{\Omega T}

\newcommand{\bOm}{{\Omega\hspace{-7pt}\Omega}}
\newcommand{\bD}{{\mathbb{D}}}

\numberwithin{equation}{section}

\begin{document}

\title{On the strong DR/DZ equivalence conjecture}

\author[X.~Blot]{Xavier Blot}
\address{X.~B.: Korteweg-de Vriesinstituut voor Wiskunde, Universiteit van Amsterdam, Postbus 94248, 1090GE Amsterdam, Nederland}
\email{x.j.c.v.blot@uva.nl}	

\author[D.~Lewa\'nski]{Danilo Lewa\'nski}
\address{D.~L.: Dipartimento di Matematica, Informatica e Geoscienze, Universit\`a degli studi di Trieste,
	Via Weiss 2, 34128	Trieste, Italia}
\email{danilo.lewanski@units.it} 

\author[S.~Shadrin]{Sergei Shadrin}
\address{S.~S.: Korteweg-de Vriesinstituut voor Wiskunde, Universiteit van Amsterdam, Postbus 94248, 1090GE Amsterdam, Nederland}
\email{s.shadrin@uva.nl}	

\begin{abstract} We establish the Miura equivalence of two integrable hierarchies associated to a semi-simple cohomological field theory: the double ramification hierarchy of Buryak and the Dubrovin-Zhang hierarchy. This equivalence was conjectured by Buryak and further refined by Buryak, Dubrovin, Gu\'er\'e, and Rossi.
\end{abstract}

\maketitle
\tableofcontents

\section{Introduction}

\subsection{Cohomological field theories} Let $V=\langle e_1,\dots,e_N\rangle$ be a finite dimensional vector space equipped with a symmetric nondegenerate bilinear form $\eta$. The first basis vector $e_1$ is distinguished. A cohomological field theory (CohFT)~\cite{KM94} is a system of linear maps 
\begin{align}
c_{g,n}\colon V^{\otimes n} \to H^*(\oM_{g,n}),\quad 2g-2+n>0,	
\end{align}
such that
\begin{enumerate}
	\item The maps $c_{g,n}$ are equivariant with respect to the $S_n$-action permuting the factors of~$V$ in $V^{\otimes n}$ and the labels of the marked points in $\oM_{g,n}$.
	
	\item $\pi^* c_{g,n}( \otimes_{i=1}^n e_{\alpha_i}) = c_{g,n+1}(\otimes_{i=1}^n  e_{\alpha_i}\otimes e_1)$, where $\pi\colon\oM_{g,n+1}\to\oM_{g,n}$ is the map that forgets the last marked point. Moreover, $c_{0,3}(e_{\alpha}\otimes e_\beta \otimes e_1) =\eta(e_\alpha\otimes e_\beta) =:\eta_{\alpha\beta}$.
	
	\item $\gl^* c_{g,n}( \otimes_{i=1}^{n} e_{\alpha_i}) = \eta^{\mu \nu}c_{g_1,n_1+1}(\otimes_{i\in I_1} e_{\alpha_i} \otimes e_\mu)\otimes c_{g_2,n_2+1}(\otimes_{j\in I_2} e_{\alpha_j}\otimes e_\nu)$, where $I_1 \sqcup I_2 = \{1,\dots,n_1+n_2\}$, $|I_1|=n_1$, $|I_2|=n_2$, $g_1+g_2=g$, and $\gl\colon\oM_{g_1,n_1+1}\times\oM_{g_2,n_2+1}\to \oM_{g,n}$ is the corresponding gluing map.
	
	\item $\gl^*c_{g,n}( \otimes_{i=1}^{n} e_{\alpha_i}) = \eta^{\mu \nu}c_{g-1,n+2}(\otimes_{i=1}^{n} e_{\alpha_i} \otimes e_\mu \otimes e_\nu)$, where $\gl\colon\oM_{g-1,n+2}\to\oM_{g,n}$ is the corresponding gluing map.
\end{enumerate}
CohFTs capture the algebraic properties of enumerative problems in different contexts, among which Gromov-Witten theory. They are fully classified in the case $c_{0,3}$ defines a structure of a semi-simple Frobenius algebra on $V$ in~\cite{Tel}. In particular, under this assumption they consist of tautological classes. 

One associates to a CohFT its potential 
\begin{align}
	F\coloneqq 
    \sum_{\substack{g,n \ge 0 \\ 2g - 2 + n > 0 \\ 1 \le \alpha_1, \dots, \alpha_n \le N \\ d_1, \dots, d_n \ge 0}}
    \frac{\eps^{2g}}{n!}\left(\int_{\oM_{g,n}}c_{g,n}(\otimes_{i=1}^n e_{\alpha_i})\prod_{i=1}^n\psi_i^{d_i}\right)\prod_{i=1}^n t^{\alpha_i,d_i}	
\end{align}
and a partition function $Z\coloneqq \exp(\epsilon^{-2}F)$. Here $\psi_i$ are the tautological classes defined as first Chern classes of the cotangent bundle on the $i$-th marked point, $t^{\alpha_i, d_i}$ are formal variables with $\alpha_i = 1, \dots, N$ and $d_i \ge 0$.

\subsection{Dubrovin-Zhang hierarchy} One can associate to a semi-simple CohFT a Hamiltonian integrable hierarchy of hydrodynamic type. This construction was initially developed by Dubrovin and Zhang~\cite{DZ}, and subsequently by Liu, Wang, and Zhang~\cite{LiuWangZhang}, in the framework of calibrated semi-simple Frobenius manifolds. In the context of CohFTs, Frobenius manifold refers to an extra homogeneity property.

A more general approach can be applied to any semi-simple CohFT and it is developed in~\cite{BPS1,BPS2}. In a nutshell, one can define a new system of coordinates
\begin{align}
\label{eq:def-w-top}
	w^{\alpha,d}\coloneqq \partial_x^d \eta^{\alpha\beta} \frac{\partial^2 F}{\partial t^{\beta,0} t^{1,0}}\big|_{t^{1,0}\to t^{1,0}+x}.
\end{align} 
It is a formal power series in $t^{\gamma,p}+\delta^\gamma_1\delta^p_0x$, $\gamma=1,\dots,N$, $p\geq 0$. The shift of $t^{1,0}$ by $x$ is needed in what follows to have a spatial variable in the construction of an integrable system.  
The second upper index of $w^{\alpha,d}$ can be omitted in case it is zero. 

The double derivatives of $F$ are proved to be differential polynomials in these new coordinates,
\begin{align}
	\frac{\partial^2 F}{\partial t^{\alpha,p}\partial t^{\beta,q}}\big|_{t^{1,0}\to t^{1,0}+x} = \Omega_{\alpha,p;\beta,q}(\{w^{\gamma,d}\}),
\end{align}
and the following evolutionary integrable system possesses a polynomial Hamiltonian structure of hydrodynamic type with a Poisson bracket $K^{\mathrm{DZ}}=\eta^{\alpha\beta}\d_x + O(\epsilon^2)$ and a tau structure:
\begin{align}
	\frac{\partial w^\alpha}{\partial t^{\beta,q}} = \partial_x \eta^{\alpha\gamma}\Omega_{\gamma,0;\beta,q}.
\end{align}

This system is described explicitly in a number of examples: for the trivial CohFT it is the KdV hierarchy~\cite{wit-1,kon}, see also~\cite{AHIS}, and for the Witten $r$-spin CohFT it is the Gelfand-Dickey hierarchy~\cite{wit-2,FSZ}. 

\subsection{Buryak's double ramification hierarchy} Buryak constructed a different integrable system associated to a CohFT in~\cite{Bur}. Let $u^1,\ldots,u^N$ be formal dependent variables, and let $u^{\alpha,d}\coloneqq \d_x^d u^\alpha$. Define differential polynomials $P^\alpha_{\beta,d}$ by
\begin{align}
	P^\alpha_{\beta,d}& \coloneqq \eta^{\alpha\gamma} \sum_{\substack{g,n\geq 0,\,2g+n>0\\k_1,\ldots,k_n\geq 0\\\sum_{j=1}^n k_j=2g}} \frac{\eps^{2g}}{n!}
	\prod_{j=1}^n u^{\alpha_j,k_j} \Coef_{a_1^{k_1}\cdots a_n^{k_n}} I^{g,d}_{\gamma,\beta,\alpha_1,\dots,\alpha_n}, 
\end{align}
where $I^{g,d}_{\gamma,\beta,\alpha_1,\dots,\alpha_n}$ is a polynomial of degree $2g$ in the variables $a_1,\dots,a_n$ given by
\begin{align}
	I^{g,d}_{\gamma,\beta,\alpha_1,\dots,\alpha_n} \coloneqq \left(\int_{\oM_{g,n+2}}
	\lambda_g
	{\DR_g\bigg(0,-\sum_{j=1}^n a_j,a_1,\ldots,a_n\bigg)}  \psi_1^d c_{g,n+2}\Big(e_\beta\otimes e_\gamma \otimes \bigotimes_{j=1}^n e_{\alpha_j}\Big) \right).
\end{align}
Here $\lambda_g$ and $\DR_g$ are the top Chern class of the Hodge bundle and the double ramification cycle, respectively. Buryak's DR hierarchy~\cite{Bur} is the following integrable system of evolutionary PDEs:
\begin{align}
	\frac{\d u^\alpha}{\d t^{\beta,d}}=\d_x P^\alpha_{\beta,d}.
\end{align}
It possesses the Hamiltonian structure given by $\eta^{\alpha\beta}\d_x$ and a tau structure. Note, however, that $u^\alpha$ are not in general the normal coordinates for the tau structure. 

The further properties of this integrable system and its generalizations are studied in a number of papers, see~\cite{BR16,BDGR1,BDGR20,BRS21}. The importance of the Buryak's approach to integrable systems associated to semi-simple CohFTs is in particular justified by that fact that it admits a quantization with respect to its Poisson structure, see~\cite{BR16-quantum,BDGR20}.

\subsection{Normal Miura equivalence} In general the Miura group acts on (Hamiltonian) integrable systems by the diffeomorphic changes of dependent variables, see e.g.~\cite{DZ}. More refined are the so-called normal Miura transformations that transform a tau-symmetric Hamiltonian hierarchy written in normal coordinates to a tau-symmetric Hamiltonian hierarchy written in normal coordinates, see e.g. \cite[Section 3]{BDGR1}.

Let $\tilde u^\alpha$ be the normal coordinates for Buryak's double ramification hierarchy \cite[Section 7]{BDGR1}. The normal coordinates for the Dubrovin-Zhang hierarchy are $w^\alpha$ \cite{DZ, BPS1}. The following conjecture is proposed in~\cite[Conjecture 7.3]{BDGR1}:

\begin{conjecture}[Strong DR/DZ equivalence conjecture] \label{conj:main} For any semi-simple cohomological field theory there exists a differential polynomial $\mathcal{P}$ in coordinates $w^{\alpha}$ (whose explicit expression is given in \cite[Theorem 4.4]{BS22}, but omitted here) such that the normal Miura transformation defined as 
	\begin{align}
		w^\alpha \mapsto \tilde u^\alpha (w) = w^\alpha + \eta^{\alpha\mu} \d_x \{\mathcal{P}, \bar h^{\mathrm{DZ}}_{\mu,0}\}_{K^{DZ}}
	\end{align} 
maps the Dubrovin-Zhang hierarchy to Buryak's double ramification hierarchy written in the normal coordinates.
\end{conjecture}

In the statement of the conjecture $ \bar h^{\mathrm{DZ}}_{\mu,0}$ is the Hamiltonian of the time flow $\partial_{t^{\mu,0}}$ for the Dubrovin-Zhang hierarchy; $\{\bullet,\bullet\}_{K^{DZ}}$ denotes the Poisson bracket of the Dubrovin-Zhang hierarchy. This conjecture refines an earlier conjecture of Buryak~\cite{Bur} on Miura equivalence of the Dubrovin-Zhang and double ramification hierarchies. 

The main application of our work is the following theorem:

\begin{theorem}\label{thm:main} Conjecture~\ref{conj:main} holds.
\end{theorem}

\subsection{Generalizations and approach to the proof}
There is a number of generalizations of the strong DR/DZ equivalence conjecture for semi-simple CohFTs above. In particular, the condition of semi-simplicity can be dropped; however, one then has to prove that the Dubrovin-Zhang integrable system is given by differential polynomials in the dependent variables. Similar constructions also apply to the so-called partial CohFTs and F-CohFTs. In the case of F-CohFTs, we get just a system of conservation laws rather than a Hamiltonian system. These constructions are available in~\cite{ABLR21,BR21, BS22}, and the corresponding conjectures there state that
\begin{enumerate}
	\item On the Dubrovin-Zhang side, the corresponding integrable system is represented by differential polynomials.
	\item The Dubrovin-Zhang integrable system is normal Miura equivalent to the double ramification system. 
\end{enumerate}
We refer to~\cite[Section 4]{BS22} for the precise statement and further details. In~\cite{BS22}, following and generalizing earlier results in~\cite{BDGR1,BGR19}, these conjectures are reduced to a system of conjectural tautological relations in the moduli space of curves known as a generalization of the $A=B$ relation. The initial $A=B$ relation itself was proposed in~\cite{BGR19} and designed to imply Conjecture~\ref{conj:main}. It is convenient to let the term ``$A=B$ relation'' refer to the whole system of (generalized) $A=B$ relations.

In this paper we prove the $A=B$ relation as well as all its generalization from~\cite{BS22} in the Gorenstein quotient. In other words we prove the relations after intersection with tautological classes of the complementary dimension (Theorem~\ref{thm:A=B}). As a consequence, our result implies the entire system of conjectures for not-necessary semi-simple CohFTs, as well as partial CohFTs and F-CohFTs under the assumption that the classes involved in these structures are tautological. Since this is always the case for semi-simple CohFTs, as a corollary Theorem~\ref{thm:A=B} implies Conjecture~\ref{conj:main} (Theorem~\ref{thm:main}). We remark that Theorem~\ref{thm:A=B} is the main result of this paper, and we moreover refer further to~\cite[Section 4]{BS22} for a list of corollaries descending from it.

\subsection{Structure of the paper} This paper is by no means self-closed as it is based on an enormous amount of prior work; we expect the reader to be familiar with the theory of (infinite-dimensional) integrable systems of evolutionary type along the lines of~\cite{DZ} and~\cite{BDGR1} as well as the intersection theory of the moduli spaces of curves and the structure of its tautological ring, along the lines of~\cite{GraPan} and~\cite{BSSZ}. Moreover, it might be instructive to follow the big steps in the literature where the conjecture of Buryak and its generalizations were further developed, which we summarise in the following.
\begin{itemize}
	\item Buryak's double ramification hierachy and the original conjecture are formulated in~\cite{Bur}; the conjecture is proved there for the trivial CohFT;
	\item The conjecture is refined to be an explicitly presented normal Miura transformation in~\cite{BDGR1}; the strong DR/DZ equivalence is proved there for the trivial CohFT;
	\item The refined conjecture is reduced to a system of relations in the tautological ring of the moduli space of curves in~\cite{BGR19};
	\item The tautological relations is further generalized to a form that we study here and is connected to a number of more general settings explained in~\cite{BS22};
	\item The relations expressions are slightly simplified in~\cite{BLRS}, in line with a twin system of relations that involve the so-called $\Omega$-classes;
	\item Some new combinatorial insights, that in particular allowed to prove the twin system of tautological relations, are obtained in~\cite{BLS-Omega};
	\item Finally, here we prove the desired relations in the Gorenstein quotient. 
\end{itemize}

Starting from this point we only talk about classes in the tautological ring of the moduli space of curves. In Section~\ref{sec:A=B} we introduce the conjectural tautological relations, the generalized $A=B$ relations, in a bit more compact form than they were presented before, and recall some basic facts about their intersections with $\psi$-classes that follow from the strong DR/DZ equivalence for the trivial CohFT. 

In Section~\ref{sec:master-relation} we introduce a new system of conjectural tautological relations, the so-called master relations, which are strongly inspired by our work on the twin system of tautological relations and localization techniques. We also make some steps towards establishing their vanishing in the Gorenstein quotient. 

Finally, in Section~\ref{sec:proof} we use a combinatorial relation between the generalized $A=B$ relations and the master relations in two different directions. In one direction, we show that the generalized $A=B$ relations are equivalent to the master relations, and in particular this equivalence holds in the Gorenstein quotient. And in the other direction, we transfer the properties of $A=B$ cycles implied by the DR/DZ correspondence for the trivial CohFT to complete the proof that the master relations vanish in the Gorenstein quotient. This proves that the generalized $A=B$ relations hold in the Gorenstein quotient.

\subsection{Acknowledgments} X.~B. and S.~S. were supported by the Netherlands Organization for Scientific Research. 
D.~L. is supported by the University of Trieste, by the INFN under the national project MMNLP (APINE) section of Trieste, by the INdAM group GNSAGA, and by the PRIN project 2022 “Geometry of algebraic structures: moduli, invariants, deformations”. The authors thank A.~Buryak, P.~Rossi, and A.~Sauvaget for useful discussions and collaboration on closely related topics. The authors also thank an anonymous referee for many useful remarks and questions that allowed to substantially improve the presentation.

\subsection{A further development} About a month after this paper was completed and submitted, the conjecture on the master relation that we pose here, Conjecture~\ref{conj:master-relation}, was proved in~\cite{BSS-Master} using the virtual localization technique on the space of stable relative maps to the projective line. 

This is an exciting development, and, first of all, in the framework of the present paper it provides an alternative to the argument that we give in Section~\ref{sec:proof-master-Gor}. Second, by Theorem~\ref{thm: equivalence Xi=0 and A=B}, this establishes the generalized \( A = B \) relations (Conjecture~\ref{conj:A=B}), thereby establishing the most general version of the DR/DZ equivalence. As a consequence, the semi-simplicity assumption in Theorem~\ref{thm:main} can be dropped. More generally, for any partial CohFT or F-CohFT, we get the following:

\begin{itemize}
  \item The Dubrovin--Zhang hierarchy exists; that is, the associated equations are differential polynomial, cf.~\cite[Theorem 4.7]{BS22}.
  \item The Dubrovin--Zhang hierarchy is equivalent to the double ramification hierarchy, cf.~\cite[Theorem 4.10]{BS22}.
\end{itemize}


However, we strongly believe that for the integrable systems community our original proof of formally weaker Theorem~\ref{thm:master-relation} that is contained in Section~\ref{sec:proof-master-Gor} is probably more appealing and might be more useful. First of all, it does not involve any additional techniques or concepts from algebraic geometry that were not already present in the construction of the DR hierarchy. Second it is much more simple and still sufficient for the proof of Conjecture~\ref{conj:main} in its original formulation. And finally, it is visibly better aligned with the methods and ideas used in the construction of the Durbovin-Zhang hierarchies in~\cite{BPS1} in the semi-simple case and thus might be used as a tool to treat the existence of the DR form of for the DZ hierarchies as a system of their universal properties.

\section{The \texorpdfstring{$A=B$}{A=B} relation in the Gorenstein quotient}

\label{sec:A=B}

\subsection{Basic notation for trees}
Let $\SRT_{g,n,m}$ be the set of stable rooted trees of total genus $g$, with $n$ regular legs $\sigma_1,\dots,\sigma_n$ and $m$ extra legs $\sigma_{n+1},\dots,\sigma_{n+m}$, which we refer to as ``frozen'' legs and must always be attached to the root vertex. For a $T\in \SRT_{g,n,m}$ we use the following notation:
\begin{itemize}
	\item $H(T)$ is the set of half-edges of $T$.
	\item $L(T),L_r(T),L_f(T)\subset H(T)$ are the sets of all, regular, and frozen legs of $T$, respectively. $L(T) = L_r(T)\sqcup L_f(T)$.
	\item $H_e(T)\coloneqq H(T)\setminus L(T)$.
	\item $\iota\colon H_e(T)\to H_e(T)$ is the involution that interchanges the half-edges that form an edge.
	\item $E(T)$ is the set of edges of $T$, $E\cong H_e(T)/\iota$.
	\item $H_+(T)\subset H(T)$ is the set of the so-called ``positive'' half-edges that consists of all regular legs of $T$ and of half-edges in $H(T)\setminus L(T)$ directed away from the root at the vertices where they are attached,
	$H_+(T)\cong E(T)\cup L_{r}(T)$; 
	\item $H_-(T)\subset H(T)$ is the set of the so-called ``negative'' half-edges that consists of all frozen legs of $T$ and of half-edges in $H(T)\setminus L(T)$ directed towards the root at the vertices where they are attached, $H_-(T)\cong E(T)\cup L_{f}(T)$;
	\item $V(T),V_{nr}(T)$ are the sets of vertices and non-root vertices of $T$. 
	\item $v_r\in V(T)$ is the root vertex of $T$; $V(T)=\{v_r(T)\}\sqcup V_{nr}(T)$.
	\item For a $v\in V(T)$, $H(v),H_+(v),H_-(v)$ are all, positive, and negative half-edges attached to $v$, respectively. Obviously, $|H_-(v_r)|=m$ and for any $v\in V_{nr}(T)$ we have $|H_-(v)|=1$.
	\item For a $v\in V(T)$ let $g(v)\in \ZZ_{\geq 0}$ be the genus assigned to $v$. The stability condition means that 
	\[\chi(v)\coloneqq 2g(v)-2+|H(v)|>0.\] 
	The genus condition reads
	\[
	\sum_{v\in V(T)} g(v) = g.
	\]
	\item We say that a vertex or a (half-)edge $x$ is a descendant of a vertex or a (half-)edge $y$ if $y$ is on the unique path connecting $x$ to $v_r$. For instance, for an edge $e$ formed by two half-edges $h_+\in H_+(T)$ and $h_-\in H_-(T)$ we assume that $e$, $h_-$, and $h_+$ are all descendants of $h_+$, $e$ and $h_-$ are both descendants of themselves and each other, and $h_+$ is not a descendant of either $e$ or $h_-$.
	\item For an $h\in H_+(T)$ let $DL(h)$ be the set of all legs that are descendants to $h$, including $h$ itself. Note that $DL(h)\subseteq L_r(T)$ for any $h\in H_+(T)$ and $DL(l)=\{l\}$ for $l\in L_r(T)$. 
	\item For an $h\in H_+(T)$ let $DH(h)$ be the set of all positive half-edges that are descendants to $h$, \emph{excluding} $h$. For instance, for $l\in L_r(T)$ we have $DH(l) = \emptyset$, and for $h\in H_+(T)\setminus L_r(T)$ we have $DH(h) \supseteq DL(h)$. 
	\item For an $e\in E(T)$ let $DL(e)$ be the set of all legs that are descendants to $e$. Note that $DL(e)\subseteq L_r(T)$ for any $e\in E(T)$. 
	\item For an $v\in V(T)$ let $DL(v)$ be the set of all regular legs that are descendants to $v$. In particular, $DL(v_r) = L_r(T)$. 
	\item For a $v\in V(T)$ let $DV(v)\subset V(T)$ be the subset of all vertices that are descendants of $v$, including $v$ itself. For instance, $DV(v_r) = V(T)$. Let 
	\[
	D\chi(v)\coloneqq \sum_{v'\in DV(v)} \chi(v').
	\]
\end{itemize}

Consider the polynomial ring $Q\coloneqq \QQ[a_1,\dots,a_n]$ and define $a\colon H_+(T) \to Q$, $a\colon E(T)\to Q$, and $a\colon V(T)\to Q$ (abusing notation we use the same symbol $a$ for all these maps) by 
\begin{align*}
	a(\sigma_i)& \coloneqq a_i, &i=1,\dots,n; 
	& & a(h)& \coloneqq \textstyle\sum_{l\in DL(h)} a(l), & h\in H_+(T); \\
	a(e)& \coloneqq \textstyle\sum_{l\in DL(e)} a(l), & e\in E(T); 
	& & a(v)& \coloneqq \textstyle\sum_{l\in DL(v)} a(l), & v\in V(T).
\end{align*} 

\subsection{Trees and strata}

In general, stable graphs are used to represent natural strata in the moduli spaces of curves, where the vertices correspond to the irreducible components, legs to the marked points, and edges to the nodes. In our setting, to each stable rooted tree $T\in\SRT_{g,n,m}$ we associate the moduli space 
\[
\overline{\mathcal{M}}_{T}=\prod_{v\in V(T)}\overline{\mathcal{M}}_{g\left(v\right),\left|H\left(v\right)\right|}.
\]
There is a canonical map, called the \emph{boundary map}, 
\[
b_{T}:\overline{\mathcal{M}}_{T}\rightarrow\overline{\mathcal{M}}_{g,n+m}
\]
whose image is the closure of the boundary stratum associated to the graph $T$. More details can be found in \cite[Sections 0.2 and 0.3]{PPZ}.

The stable rooted trees in $\SRT_{g,n,m}$ are used below to represent strata in $\oM_{g,n+m}$, and the extra combinatorial structure that we introduce here is used to specify the classes that we study.

\subsection{Leveled stable rooted trees} \label{sec:LeveledSRT}

We enhance the structure of a stable rooted tree to what we call a degree-labeled stable rooted tree (of genus $g$, with $n$ regular and $m$ frozen legs). To this end we take a stable rooted tree $T\in \SRT_{g,n,m}$ and assign to each $v\in V(T)$ an extra degree label $p(v)\in \ZZ_{\geq 0}$ such that $p(v) \leq 3g(v)-3+|H(v)|$. Let $\mathcal{P}(T)$ denote the set of degree label functions on $T$. A degree-labeled stable rooted tree is a pair $(T,p)$, where $T\in\SRT_{g,n,m}$ and $p\in\mathcal{P}$.

\begin{remark}
    The nonnegative integer $p\left(v\right)$ represents the total degree of the cohomology classes on $\overline{\mathcal{M}}_{g\left(v\right),\left|H\left(v\right)\right|}$. 

\end{remark}

Our next goal is to assign to a degree-labeled stable rooted tree $(T,p)$ a so-called admissible level function. A function $\ell\colon V(T)\to\ZZ_{\geq 0}$ is called an admissible level function if the following conditions are satisfied:
\begin{itemize}
	\item The value of $\ell$ on the root vertex is zero ($\ell(v_r) = 0$).
	\item If $v'\in DV(v)$ and $v'\not=v$, then $\ell(v')>\ell(v)$. 
	\item There are no empty levels, that is, for any $0\leq i \leq \max \ell(V(T))$ the set $\ell^{-1}(i)$ is non-empty. 
	\item For every $0\leq i\leq \max \ell(V(T))-1$ we have inequality 
	\begin{align}
            \label{eq:inequality-levels}
	    |\{v\in V(T) \,|\, \ell(v)\leq i \}|-1+\sum_{\substack{v\in V(T) \\ \ell(v)\leq i }} p(v) \leq \sum_{\substack{v\in V(T) \\ \ell(v)\leq i }} 2g(v) - 2 + m.
	\end{align}
	For instance, if $T$ has more than one vertex, then $p(v_r) \leq 2g(v_r)-2+m$. 
\end{itemize}
Let $\mathcal{L}(T,p)$ denote the set of admissible level functions on $(T,p\in\mathcal{P}(T))$. 

The set $\LDLSRT_{g,n,m}$ of leveled degree-labeled stable rooted trees consists of triples $(T,p,\ell)$, where $T\in\SRT_{g,n,m}$, $p\in\mathcal{P}(T)$, and $\ell\in \mathcal{L}(T,p)$, and it is a finite set.

\begin{example}\label{ex-graph}
Let $T\in\SRT_{g,n=5,m=2}$ be the tree shown in the following picture, where $g_{0}+g_{1}+g_{2}+g_{3}=g$. Each vertex is depicted as a circle, with its genus indicated inside. The root is the vertex of genus $g_{0}$. The two frozen legs are represented by wavy lines, while the five regular legs are represented by regular lines. 
The half-edges of $H_{+}\left(T\right)$ are decorated with $\psi$-classes. We introduce
\[
\ensuremath{p(v_{0}):=d_{1}+d_{6}+d_{7},\quad p(v_{1}):=d_{2}+d_{8},\quad p(v_{2}):=d_{3},\quad p(v_{4})=d_{4}+d_{5}}.
\]
We suppose that $p\left(v_{i}\right)\leq3g\left(v_{i}\right)-3+\left|H\left(v_{i}\right)\right|$, for $i=1,\dots,5$, so that $\left(T,p\right)$ is a degree-labeled stable rooted tree. 
A choice of a level function is represented with the dashed lines: the root is assigned level $0$, the vertex $1$ lies at level $1$, and the vertices $2$ and $3$ are at level $2$. We assume that 
\[
p\left(v_{0}\right)\leq2g_{0},\quad{\rm and}\quad p\left(v_{0}\right)+p\left(v_{1}\right)+1\leq2\left(g_{0}+g_{1}\right),
\]
so that this level function is admissible for the pair $\left(T,p\right)$. Observe that cutting the tree along the dashed line at level $i$ yields a subtree to the left, which is itself a rooted tree decorated with $\psi$-classes, and the degree of the associated cohomology class equals the left-hand side of inequality (\ref{eq:inequality-levels}).

\end{example}
\begin{center}
    \begin{tikzpicture}[
        level/.style={font=\small, inner sep=1pt, fill=white}
    ]
    
    \node[draw, circle, minimum size=1cm, inner sep=0pt] (g0) at (-0.4,0) {$g_0$};
    \node[draw, circle, minimum width=1cm, minimum height=1cm] (g1) at (2.7,1) {$g_1$};
    \node[draw, circle, minimum width=1cm, minimum height=1cm] (g2) at (5.5,1) {$g_2$};
    \node[draw, circle, minimum size=1cm, inner sep=0pt] (g3) at (3,-1.5) {$g_3$};
    
    \draw (g0) -- (g1);
    \draw (g1) -- (g2);
    \draw (g0) -- (g3);
    
    \draw[decorate, decoration={snake, segment length=4pt}] (-1.2,0.5) -- (g0);
    \draw[decorate, decoration={snake, segment length=4pt}] (-1.2,-0.5) -- (g0);
    
    \draw (g0) -- ++(0.8,0.8);
    \draw (g1) -- ++(0.8,0.8);
    \draw (g3) -- ++(1,0.5);
    \draw (g3) -- ++(1,-0.5);
    \draw (g2) -- ++(1,0);
    
    \draw[dashed] plot [smooth, tension=1] coordinates {(-0.9,-2.5) (-0.5,-1.7) (1,-0.3) (1,1.8) (1,2.3)};
    \draw[dashed] plot [smooth, tension=1] coordinates {(1,-2.5) (1.5,-1) (3.5,0.5) (4,2.3)};
    
    \node[level] at (-0.3,2.3) {Level 0};
    \node[level] at (2.3,2.3) {Level 1};
    \node[level] at (5,2.3) {Level 2};
    
    \node at (-0.1,-2.5) {\scriptsize $\leq 2g_0 - 2$};
    \node at (2.2,-2.5) {\scriptsize $\leq 2(g_0 + g_1) - 2$};
    \node at (0.5,1) {\scriptsize $\psi^{d_1}$};
    \node at (0.6,0.5) {\scriptsize $\psi^{d_6}$};
    \node at (0.6,-0.1) {\scriptsize $\psi^{d_7}$};
    
    \node at (3.7,2) {\scriptsize $\psi^{d_2}$};
    \node at (3.6,1.2) {\scriptsize $\psi^{d_8}$};
    
    \node at (4.1,-0.8) {\scriptsize $\psi^{d_4}$};
    \node at (4.1,-1.7) {\scriptsize $\psi^{d_5}$};
    
    \node at (6.5,1.2) {\scriptsize $\psi^{d_3}$};
    
    \end{tikzpicture}    
\end{center}

\subsection{Definition of the \texorpdfstring{$B$}{B} class} The goal of this section is to define the so-called $B$ class in the tautological ring of $\oM_{g,n+m}$. 

Let $(T,p,\ell)\in \LDLSRT_{g,n,m}$. Assign to each $v\in V(T)$ the moduli space of curves $\oM_{g(v),|H(v)|}$, where the first $|H_+(v)|$ marked points correspond to the positive half-edges attached to $v$ and ordered in an arbitrary but fixed way and the the last $|H_-(v)|$ marked points correspond to the negative half-edges attached to $v$, also ordered in some arbitrary but fixed way. Consider the class 
\begin{align*}
	\Psi(v)\coloneqq \left(\prod_{i=1}^{|H_+(v)|} \frac{1}{1-a(h_i)\psi_i}\right)_{p(v)} \in R^{p(v)}(\oM_{g(v),|H(v)|})\otimes_{\QQ}Q
\end{align*}

For each $(g,n.m)$ such that $2g-2+n+m>0$ define the class 
$$B^m_{g,n}\in R^*(\oM_{g,n+m})\otimes_{\QQ}Q$$
as
\begin{equation} \label{eq:Deinition-of-B}
	B^m_{g,n} \coloneqq \sum_{(T,p,\ell)\in\LDLSRT(g,n,m)} (-1)^{\max\ell(V(T))} \biggl(\prod_{e\in E(T)} a(e)\biggr) (b_T)_* \bigotimes_{v \in V(T)} \Psi(v)
\end{equation}
Here $(b_T)_*$ is the boundary pushforward map that acts from $\bigotimes_{v \in V(T)} R^*(\oM_{g(v),|H(v)|})\otimes_{\QQ} Q$ to $R^*(\oM_{g,n+m})\otimes_{\QQ}Q$. 
The class $B^m_{g,n}$ 
has the feature that its component in $R^d$ is a homogeneous polynomial of degree $d$ in $a_1,\dots,a_n$, $d=0,\dots,3g-3+n+m$.

\begin{example}
Let $\left(T,p,\ell\right)\in\LDLSRT_{g,n,m}$ be as in Example~\ref{ex-graph}. Suppose that the regular leg $\sigma_{i}$ is decorated with $\psi^{d_{i}}$, and carries the weight $a\left(\sigma_{i}\right)=a_{i}$, for $i=1,\dots,5.$ Then the coefficient of $(b_{T})_{*}\bigotimes_{v\in V(T)}\Psi(v)$ in Eq.~(\ref{eq:Deinition-of-B}) is given by
\[
\left(-1\right)^{2}\left(a_{2}+a_{3}\right)\left(a_{3}\right)\left(a_{4}+a_{5}\right).
\]

\end{example}

\begin{remark}
    A brief history of the $B$-class: the class $B^{0}$ was first introduced in \cite{BGR19}. In \cite{BS22}, this definition was generalized to $B^{m}$ for all $m\geq0$, and reshaped using level structures. Finally, \cite[Section 2.5]{BLRS}  presents the form of the $B$-class used here. This last formulation is equivalent to the one of \cite{BS22} by \cite[Theorem 2.15]{BLRS}. Notice that in \cite{BS22} the $B$-class is indexed with multi-indices $(d_1,\dots,d_n)$, which corresponds to extracting the coefficient of $a_1^{d_1}\cdots a_n^{d_n}$ in the above definition of $B_{g,n}^m$.
\end{remark} 

\subsection{The conjectural \texorpdfstring{$A=B$}{A=B} relation revisited}

Let $a_1,\dots,a_n$ be a list of positive integers. Consider the moduli space 
\begin{align}
\overline{\mathcal{M}}_{g}^{\sim}(\mathbb{P}^1,a_1,\dots,a_n,-\sum_{i=1}^n a_i)	
\end{align}
of rubber stable maps to $(\mathbb{P}^1,0,\infty)$. Let $s\colon \oM_{g}^{\sim}(\mathbb{P}^1,a_1,\dots,a_n,-\sum_{i=1}^n a_i)\to \oM_{g,n+1}$ be the projection to the source curve, and $\lambda_g$ the lift of the lambda class with respect to this projection. Let $t\colon \oM_{g}^{\sim}(\mathbb{P}^1,a_1,\dots,a_n,-\sum_{i=1}^n a_i)\to LM_{2g-1+n}$ be the projection to the target curve, where $LM_{2g-1+n}$ denotes the Losev-Manin space (see e.g. \cite{BLRS} for more details) with $2g-1+n$ marked points. Let $\tilde \psi_0$ be the pull-back by $t$ of the $\psi$-class at the point $0$ in Losev-Manin space. Define
\begin{align} \label{eq:NewA1}
	A^1_{g,n} \coloneqq s_* \left(\frac{\lambda_g}{1-\tilde\psi_0} \left[\oM_{g}^{\sim}\left(\mathbb{P}^1,a_1,\dots,a_n,-\sum\nolimits_{i=1}^n a_i\right)\right]^{\mathrm{vir}}\right)\in R^*(\oM_{g,n+1})\otimes_{\QQ}Q.
\end{align}

Note that $(A^1_{g,n})_{<2g}=0$ and $(A^1_{g,n})_{2g}=\lambda_g\DR_{g}(a_1,\dots,a_n,-\sum_{i=1}^n a_i)$. Note also that for any $d\geq 2g$ the class $(A^1_{g,n})_{d}$ is a homogeneous polynomial of degree $d$ in $a_1,\dots,a_n$. This follows, for example, from the computation of the class $\tilde{\psi}_0$ in the proof Lemma~\ref{lem new def A}.

\begin{remark}
    A brief history of the $A$-class: the class $A^0_{g,d_1,\dots,d_n} \in H^*(\oM_{g,n})$ was introduced in \cite{BDGR20,BGR19} as a sum over stable trees decorated with $\lambda_g$-classes and DR-cycles. In a similar spirit, the class $A^1_{g,d_1,\dots,d_n} \in H^*(\oM_{g,n+1})$ was introduced in \cite{BS22}, where it was observed that $A^0_{g,d_1,\dots,d_n}=\frac{1}{\sum a_i} \pi_*\left(A^1_{g,d_1,\dots,d_n} \right)$, where $\pi\colon \oM_{g,n+1}\to\oM_{g,n}$ is the map that forgets the last marked point. The class $A^1_{g,n}$ defined in Eq.~(\ref{eq:NewA1}) provides a new formulation of the class $A^1$; as shown in Lemma~\ref{lem new def A} below this formulation satisfies $A^1_{g,d_1,\dots,d_n} = \mathrm{Coeff}_{a_1^{d_1}\cdots a_n^{d_n}}A_{g,n}^1$, where $\mathrm{Coeff}_{a_1^{d_1}\cdots a_n^{d_n}}$ denotes the operation of extracting the coefficient of the monomial $a_1^{d_1}\cdots a_n^{d_n}$.
\end{remark}

\begin{conjecture}[generalized $A=B$ relations]\label{conj:A=B} The following three statements hold:
	\begin{enumerate}
		\item For any $(g,n,m)$ such that $g\geq 0$, $n\geq 1$, $m\geq 2$, we have 
		\begin{align}
					\deg B^m_{g,n} \leq 2g-2+m.
		\end{align}
		\item For any $(g,n)$ such that $g\geq 0$, $n\geq 1$, $2g-1+n>0$, we have 
		\begin{align}
			\deg (B^1_{g,n} - A^1_{g,n})\leq 2g-1.
		\end{align}
			\item For any $(g,n)$ such that $g\geq 0$, $n\geq 1$, $2g-2+n>0$, we have 
		\begin{align}
			\deg (B^0_{g,n} - \frac{1}{\sum_{i=1}^n a_i} \pi_*A^1_{g,n})\leq 2g-2,
		\end{align}
		where $\pi\colon \oM_{g,n+1}\to\oM_{g,n}$ is the map that forgets the last marked point.
	\end{enumerate} 
\end{conjecture}
In all three statement $\deg$ refers either to cohomological degree or to the homogenenous degree in $a_1,\dots,a_n$. We use the convention that a class $\alpha \in H^{2i}(X)$ has cohomological degree $i$ instead of $2i$ and that a polynomial of negative degree is the zero polynomial. Note that the third statement follows from the second statement. We have the following lemma:

\begin{lemma} \label{lem new def A}Conjecture~\ref{conj:A=B} is equivalent to \cite[Conjectures~1, 2, and 3]{BS22}.
\end{lemma}

\begin{proof} We have to show that our definition~(\ref{eq:NewA1}) of the class $A^1_{g,n}$ coincides with the definition given in~\cite{BS22} (see also~\cite[Section 2.4]{BLRS} for notations matching ours). To this end, we will first perform the intersection of the class $\tilde\psi_0$ with the moduli space of rubber maps using the technics of~\cite{BSSZ}, as we will show this yields exactly the same trees in $\SRT_{g,n,1}$ decorated by double ramification cycles as in the original definition of $A^1_{g,n}$. The only difference is the combinatorial coefficient that we will compare in a second time.

We express powers of $\psi_{0}$ in the Losev-Manin space $LM_{2g-1+n}$ using its expression as a boundary divisor \cite[Eq. (2)]{BSSZ} to obtain
\[
\psi_{0}^{d}=\sum_{\substack{\chi_{1}+\cdots+\chi_{d+1}=2g-1+n\\ \chi_{1},\dots,\chi_{d+1}>0}}D_{\chi_{1},\dots,\chi_{d+1}}\frac{\chi_{1}}{\chi_{1}+\cdots+\chi_{d+1}}\frac{\chi_{2}}{\chi_{2}+\cdots+\chi_{d+1}}\cdots\frac{\chi_{d+1}}{\chi_{d+1}},
\]
where $D_{\chi_{1},\dots,\chi_{d+1}}$ denotes the boundary strata in $LM_{2g-1+m}$ composed of $d+1$ components, where we number the components such that the first component contains $\infty$, the second is attached to the first and so on, and such that the $i$th component contains $\chi_{i}$ marked points (excluding the points $0$ and $\infty$ on each component). We now compute $\lambda_{g}s_{*}t^{*}\left(D_{\chi_{1},\dots,\chi_{d+1}}\right)$ using the following two observations. First, since the expression is ultimately multiplied by $\lambda_{g}$, we can consider that the prestable curves involved in $t^{*}\left(D_{\chi_{1},\dots,\chi_{d+1}}\right)$ have no loop in their dual graph. Second, by \cite[Lemma 2.3]{BSSZ}, the ramifications points of total order $\chi_i$ lying above the $i$th  component of $D_{\chi_{1},\dots,\chi_{d+1}}$ (excluding those above the nodes) must belong to a unique component, although additional unstable components may still be present. Note that by the Riemann-Huwtiz formula, the genus $g_i$ and the number of nodes $n_i$ of this stable component satisfy $\chi_i = 2g_i-2+n_i$.  As a result $\lambda_{g}s_{*}t^{*}\left(D_{\chi_{1},\dots,\chi_{d+1}}\right)$ gives a sum over trees $T \in \SRT_{g,n,1}$, where the contribution of each tree coincides precisely with the term $A\left(T\right)$ in the expression of the $A^1_{g,n}$-class from \cite[Section 2.4]{BLRS}, with combinatorial coefficient equal to $1$.

In the case of the original definition, the combinatorial coefficient of a tree $T\in \SRT_{g,n,1}$ is equal to 
\begin{align}
C_1(T) = \prod_{v\in V(T)} \frac{\chi(v)}{D\chi(v)}.	
\end{align}
For the computation of~\eqref{eq:NewA1} we have to sum over the number of times the same tree $T$ appears in the computation of powers of $\tilde{\psi}_{0}$ . It is easy to see that this corresponds to summing over the set of rigorous level structures $\mathcal{R}\left(T\right)$, that is, the set of one-to-one maps $f\colon V(T)\to\{1,\dots,|V(T)|\}$ such that $f(v_{1})>f(v_{2})$ if $v_{1}\in DV(v_{2})$. Then  the coefficient implied by~\eqref{eq:NewA1} is equal to
\begin{align}
	C_2(T) = \sum_{f\in \mathcal{R}(T)} \prod_{i=1}^{|V(T)|} \frac{\chi(f^{-1}(i))}{\chi(f^{-1}(i))+\cdots+\chi(f^{-1}(|V(T)|))}.	
\end{align}
Let us show that $C_1(T)=C_2(T)$. To this end, observe that both the definition of $C_1$ and $C_2$ can be applied to forests, moreover, 
\begin{align}
	C_1(T_1\sqcup \cdots T_k) = \prod_{i=1}^k C_1(T_i) \quad \text{and} \quad C_2(T_1\sqcup \cdots T_k) = \prod_{i=1}^k C_2(T_i).
\end{align}
Moreover, the factor in $C_1(T)$ corresponding to $v_r$ coincides with the factor for $v_r$ in each summand in $C_2(T)$. Thus we can remove the root vertex and use the factorization for the forests above to prove $C_1(T)=C_2(T)$ by induction on the size of the tree. 
\end{proof}

Various simplifications and partial results towards Conjecture~\ref{conj:A=B} are available in~\cite{BLRS, BHS,BS22,Gub22,LiuWang}.

\subsection{The \texorpdfstring{$A=B$}{A=B} identity in the Gorenstein quotient}  The main statement that we prove in this paper is the following theorem:

\begin{theorem}[main theorem]\label{thm:A=B} Conjecture~\ref{conj:A=B} holds in the Gorenstein quotient. Namely, the following three statements hold:
	\begin{enumerate}
		\item For any $(g,n,m)$ such that $g\geq 0$, $n\geq 1$, $m\geq 2$, and for any $\alpha \in R^*(\oM_{g,n+m})$ we have 
		\begin{align}
			\deg \left(\int_{\oM_{g,n+m}} \alpha B^m_{g,n}\right)  \leq 2g-2+m.
		\end{align}
		\item For any $(g,n)$ such that $g\geq 0$, $n\geq 1$, $2g-1+n>0$, and for any $\alpha \in R^*(\oM_{g,n+1})$ we have 
		\begin{align}
			\deg \left(\int_{\oM_{g,n+1}} \alpha (B^1_{g,n} - A^1_{g,n})\right)\leq 2g-1.
		\end{align}
		\item For any $(g,n)$ such that $g\geq 0$, $n\geq 1$, $2g-2+n>0$, and for any $\alpha \in R^*(\oM_{g,n})$ we have 
		\begin{align}
			\deg \left(\int_{\oM_{g,n}} \alpha (B^0_{g,n} - \frac{1}{\sum_{i=1}^n a_i} \pi_*A^1_{g,n})\right)\leq 2g-2.
		\end{align}
	\end{enumerate} 
\end{theorem}

\begin{remark} \label{rem:thethirdstatement}
The third statement is a corollary of the second one that one can obtain by the projection formula. 	
\end{remark}

As a corollary of this theorem, several important consequences for integrable hierarchies follow. These consequences are developed in detail in~\cite[Section 4]{BS22}, but we include here a summary of the key ideas for clarity.

To illustrate these implications, consider the case of a cohomological field theory $\{c_{g,n}\}$. The key observation is that the $m$-th derivative of the potential of the CohFT 
\begin{align}
    \frac{\partial^m F}{\prod_{i=1}^m \partial t^{\alpha_i,p_i}}\big|_{t^{1,0}\to t^{1,0}+x}
\end{align}
is represented as a difference of a naturally defined polynomial in the variables $\{w^{\gamma,d}\}$, introduced in Eq.~(\ref{eq:def-w-top}), and a series in the variables $\{t^{\gamma,d}+\delta^{\gamma,d}_{1,0} x\}$ whose coefficients are the intersection numbers of the type
\begin{align}
\label{eq:Bm-part-equation}
\int_{\oM_{g,n+m}} c_{g,n+m}\big(\bigotimes_{i=1}^n e_{\beta_i}\otimes \bigotimes_{j=1}^m e_{\alpha_i}\big)\prod_{j=1}^m \psi_{n+j}^{p_j} \mathrm{Coeff}_{a_1^{d_1}\cdots a_n^{d_n}}(B^m_{g,n}),
\end{align}
for arbitrary $1\leq \beta_1,\dots,\beta_n\leq N$ and arbitrary $d_1,\dots,d_n\geq 0$ such that $\sum_{i=1}^n d_i > 2g-2+m$. We refer to~\cite[Theorem 4.6]{BS22} for a more general version of this observation that works for any F-CohFT, and it is based on \cite[Proposition 7.2]{BDGR1} and \cite[Proposition 3.5]{BGR19}.

The first statement of the main theorem implies the existence of the Dubrovin–Zhang hierarchy for any F-CohFT whose classes lie in the tautological ring, in the sense that the system of conservation laws is polynomial~\cite[Theorem 4.7]{BS22}. In the special case of a CohFT whose classes belong to the tautological, the vanishing of (\ref{eq:Bm-part-equation}) for $m=2$ directly implies the polynomiality of the conservation laws. 

The second statement implies that the Dubrovin-Zhang hierarchy is Miura equivalent, in a very explicit way, to the DR hierarchy of conservation laws associated to the same F-CohFT, also under assumption that its classes belong to the tautological ring~\cite[Theorem 4.10]{BS22}. In the vein of the discussion above, one has to notice that the integrals
\begin{align}
\int_{\oM_{g,n+1}} c_{g,n+1}\big(\bigotimes_{i=1}^n e_{\beta_i}\otimes e_{\alpha}\big) \mathrm{Coeff}_{a_1^{d_1}\cdots a_n^{d_n}}(A^1_{g,n})
\end{align}
describe the coefficient  of a vector field that generates Buryak's DR hierarchy~\cite[Theorem 4.9]{BS22}, and thus the corresponding vanishing statement identifies the conservation laws of the the DR hierarchy with the ones of the DZ hierarchy, up to a polynomial correction in the variables $\{w^{\gamma,d}\}$. The $\partial_x$ derivative of the latter correction defines the Miura transformation of the two hierarchies~\cite[Theorem 4.10 and Equation~(4.12)]{BS22} (in the latter reference it is explained in the natural generality of F-CohFTs).

Finally, the third statement implies the strong DR/DZ correspondence for (partial) CohFTs, also under assumption that the classes of (partial) CohFT belong to the tautological ring. In the vein of the discussion above, one has to notice that the integrals
\begin{align}
\int_{\oM_{g,n}} c_{g,n}\big(\bigotimes_{i=1}^n e_{\beta_i}\otimes e_{\alpha}\big) \mathrm{Coeff}_{a_1^{d_1}\cdots a_n^{d_n}}\left(\frac{1}{\sum_{i=1}^n a_i} \pi_*A^1_{g,n})\right)
\end{align}
describe the coefficients of the logarithm of the tau function of Buryak's DR hierarchy~\cite[Proposition 6.10]{BDGR1} and~\cite[Theorem 6.1]{BDGR20}.  This is a consequence of the new expression for the $A$-class, with its correspondence to the original definition detailed in Lemma~\ref{lem new def A}. Thus the corresponding vanishing statement identifies logarithms of the tau functions of the DZ and DR hierarchies, up to a polynomial correction in the variables $\{w^{\gamma,d}\}$. The derivatives of the latter correction define the normal Miura transformation that connects the two hierarchies~\cite[Section 4.4.4]{BS22} (in the latter reference it is explained in the natural generality of partial CohFTs).

We refer for further details to~\cite{BS22,BDGR1,BGR19}.

One important lemma that allows to establish the strong DR/DZ correspondence for any semi-simple cohomological field theory is a direct corollary of the classification proved in~\cite{Tel}:

\begin{lemma} [Corollary of~\cite{Tel}] The classes of any semi-simple cohomological field theory belong to the tautological ring.
\end{lemma}

\subsection{The trivial CohFT} An important case when the Dubrovin-Zhang hierarchy and the DR/DZ equivalence are fully understood is the case of the trivial CohFT~\cite{Bur,BDGR1}. In particular, this implies the following special cases of Theorem~\ref{thm:A=B}:

\begin{lemma}\label{lem:TrivialCohFT} The following three statements hold:
	\begin{enumerate}
		\item For any $(g,n,m)$ such that $g\geq 0$, $n\geq 1$, $m\geq 2$, and for any monomial $\prod_{i=1}^m \psi_{n+i}^{b_i} \in R^*(\oM_{g,n+m})$ we have 
		\begin{align}
			\deg \left(\int_{\oM_{g,n+m}} \left(\prod_{i=1}^m \psi_{n+i}^{b_i}\right)  B^m_{g,n}\right)  \leq 2g-2+m.
		\end{align}
		\item For any $(g,n)$ such that $g\geq 0$, $n\geq 1$, $2g-1+n>0$, we have 
		\begin{align} 
			\deg \left(\int_{\oM_{g,n+1}} (B^1_{g,n} - A^1_{g,n})\right)\leq 2g-1.
		\end{align}
		\item For any $(g,n)$ such that $g\geq 0$, $n\geq 1$, $2g-2+n>0$ we have 
		\begin{align}
			\deg \left(\int_{\oM_{g,n}} (B^0_{g,n} - \frac{1}{\sum_{i=1}^n a_i} \pi_*A^1_{g,n})\right)\leq 2g-2.
		\end{align}
	\end{enumerate} 
\end{lemma}

\begin{proof} The second and the third statement are literally equivalent to the statements of the strong DR/DZ correspondence in this case, see~\cite{Bur,BDGR1,BS22}. The first statement is equivalent by~\cite[Section 4]{BS22} to the fact that the second and higher derivatives of the potential $F(t_0,t_1,\dots)$ corresponding to the string solution are known to be differential polynomials in the dependent variables in this case, see e.~g.~\cite{DZ}.
\end{proof}

\begin{remark} Note that in the case $m=1$ we don't have a ready statement on the degree of $\int_{\oM_{g,n+1}} \psi_{n+1}^b (B^1_{g,n} - A^1_{g,n})$ for $b>0$. The reason for this is that the first derivatives of the potential $F$ with respect to the variables $t_b$, $b>0$, play no direct role in the construction of integrable systems in this case.
\end{remark}

\section{The master relation in the Gorenstein quotient} \label{sec:master-relation}

\subsection{Main definitions and statements} The definition and overall idea to use what we call \emph{master relation} comes from a parallel paper~\cite{BLS-Omega}, where we use a close relative of this relation to prove the conjectures on the so-called $\Omega$-classes posed in~\cite{BLRS}.

\subsubsection{Pre-stable star rooted trees}
Fix $m\geq 1$, $n\geq 1$, and $g\geq 0$. Let $\PSSRT_{g,n,m}$ be the set of pre-stable star rooted trees, that is, the rooted trees with one root vertex, where the frozen legs $\sigma_{n+1},\dots,\sigma_{n+m}$ are attached, and all other vertices are connected by an edge to the root (hence the term ``star''). We also demand that no regular legs are attached to the root vertex. 

The graphs that we obtain are quite similar to the graphs that form a subset of $\SRT_{g,n,m}$, and we extend the definitions applying the same notation for various concepts related to these graphs. However, since we allow non-stable vertices, $\PSSRT_{g,n,m}$ is not quite a subset of $\SRT_{g,n,m}$.

\subsubsection{Classes assigned to vertices}
Let $T\in \PSSRT_{g,n,m}$. As in the case of trees in $\SRT_{g,n,m}$, we assign to each $v\in V(T)$ the moduli space of curves $\oM_{g(v),|H(v)|}$, where the first $|H_+(v)|$ marked points correspond to the positive half-edges attached to $v$ and ordered in an arbitrary but fixed way and the the last $|H_-(v)|$ marked points correspond to the negative half-edges attached to $v$, also ordered in some arbitrary but fixed way. The cases of non-stable pairs $(g(v),|H(v)|)$ will be treated separately, but informally one can think of a natural extension of the classes we use for stable vertices to the unstable moduli spaces. 

For the root vertex, we consider the class 
\begin{align}
	\Psi(v_r)& \coloneqq \prod_{i=1}^{|E(T)|} \frac{1}{1-a(h_i)\psi_i} \in R^*(\oM_{g(v_r),|E(T)|+m})\otimes_{\QQ}Q
\end{align}
Note that for pre-stable star rooted trees, $H_+(v_r)$ can be identified with $E(T)$. In the exceptional unstable case $g(v_r)=0, m=1, |E(T)|=1$ we formally assign to the root vertex the following class:
\begin{align}
	a(h_1)^{-1}\in R^{-1}(\oM_{0,2})\otimes_{\QQ}Q[a(h_1)^{-1}],
\end{align}
where the negative cohomological degree and the space $R^{-1}(\oM_{0,2})$ are  just formally assigned to allow to treat this case non-exceptionally in what follows, where it is replaced by a contraction rule under the boundary pushforward map. For instance, in this vein it is often convenient to formally extend the definition of the integrals of $\psi$ classes over $\oM_{0,2}$, where they are defined by $\int_{\oM_{0,2}} \big((1 - x_1\psi_1)(1 - x_2\psi_2)\big)^{-1} \coloneqq (x_1 + x_2)^{-1}$ .

For a non-root vertex $v$, we consider the class 
\begin{align}
	\bD(v)& \coloneqq \frac{\lambda_{g(v)}\DR_{g(v)}\big(a(h_1),\dots,a(h_{|H_+(v)|}),-a(v)\big)}{1-a(v)\psi_{|H(v)|}}   \in R^*(\oM_{g(v),|H(v)|})\otimes_{\QQ}Q
\end{align}
Note that in this case there exactly $|H(v)|= |H_+(v)|+ 1$ half-edges attached to $v$ with all positive half-edges being the regular legs, and the only negative half-edge being a part of the edge connecting $v$ to $v_r$.
Note also that in this case $a(v) = 
a(h_1)+\cdots+a(h_{|H_+(v)|})$. In the exceptional unstable case $g(v)=0, H_+(v)=1$ we formally assign to this vertex the following class:
\begin{align}
	a(h_1)^{-1}\in R^{-1}(\oM_{0,2})\otimes_{\QQ}Q[a(h_1)^{-1}],
\end{align}
where, as it was for the root vertex,  the negative cohomological degree and the space $R^{-1}(\oM_{0,2})$ are just formally assigned to allow to treat this case non-exceptionally in what follows. 

\subsubsection{Classes assigned to trees}
Now, let us assume that $2g-2+n+m>0$ and assign to a tree $T\in \PSSRT_{g,n,m}$ a class $\Xi(T)\in R^*(\oM_{g,n+m})\otimes_{\QQ}Q[u,u^{-1}]$. Here $u$ is a new formal variable to control the dimension.
We have:
\begin{align}
	\Xi(T)\coloneqq u^{2g-2+m} \left(\prod_{e\in E(T)} \frac{a(e)}{u}\right) (b_T)_* \left( \left(\sum_{d=-1}^\infty \frac{\Psi(v_r)_d}{(-u)^d}\right) \otimes\bigotimes_{v\in V(T)\setminus\{v_r\}} \left( \sum_{d=-1}^\infty \frac{\bD(v)_d}{u^d} \right) \right),
\end{align} 
where $(b_T)_*$ is the boundary pushforward map from $\bigotimes_{v \in V(T)} R^{*}(\oM_{g(v),|H(v)|})$ to $R^{*}(\oM_{g,n+m})$ extended by linearity to the rational functions in $a_i$ and Laurent polynomials in $u$.  In the case of pre-stable trees $b_T$ is assumed to contract the unstable components. Note that under the assumption $2g-2+n+m>0$ the dependence of the resulting formula on $a_i$'s is purely polynomial, so $\Xi(T) \in R^{*}(\oM_{g,n+m}) \otimes_{\QQ} Q[u,u^{-1}]$. Note also that for $2g-2+n+m>0$ the dimension of $\oM_{g,n+m}$ is $3g-3+n+m$ and thus the possible range of degrees in $u$ is from $-g+1-n$ to $2g-2+m$. 

\begin{remark} Note that $\mathrm{Coeff}_{u^{2g-2+m-d}}\Xi(T)\in R^{d}(\oM_{g,n+m}) \otimes_{\QQ} Q$
	and
	$\Xi(T)_d$ is a homogeneous polynomial of degree $d$ in $a_1,\dots,a_n$. (Here and below $\mathrm{Coeff}_{u^p}\xi$ denotes the coefficient of $u^p$ in a Laurent polynomial $\xi(u)$.) 
\end{remark}

\begin{remark}
As a side remark, note that in the case $2g-2+n+m=0$, that is, $g=0,n=m=1$, we have just one tree with two vertices, both unstable, and we can extend the definition of the class $\Xi(T)$ given above to produce the class 
$a_1^{-1}u^{0} \in R^{-1}(\oM_{0,2}) \otimes_{\QQ} Q[a_1^{-1}][u,u^{-1}]$.	
\end{remark}

\subsubsection{Classes assigned to sets of trees}
Let $2g-2+n+m>0$. Consider the following Laurent polynomial in a formal variable $u$ and a polynomial in the variables $a_1,\dots,a_n$:
\begin{align} \label{eq:DefinitionOfXi}
	\Xi^m_{g,n} \coloneqq \sum_{T\in \PSSRT_{g,n,m}} \Xi(T) .
\end{align}

\begin{remark} Note that $(\Xi^m_{g,n})_d\coloneqq \mathrm{Coeff}_{u^{2g-2+m-d}}\Xi^m_{g,n}\in R^{d}(\oM_{g,n+m}) \otimes_{\QQ} Q$
	and
	$(\Xi^m_{g,n})_d$ is a homogeneous polynomial of degree $d$ in $a_1,\dots,a_n$. 
\end{remark}

\begin{conjecture}[master relation] \label{conj:master-relation} For any $g\geq 0$, $m,n\geq 1$, $2g-2+n+m>0$, 
\begin{align}
	\Xi^m_{g,n} \in R^{*}(\oM_{g,n+m}) \otimes_{\QQ} Q[u].
\end{align}
In other words, for any $d\geq 2g-1+m$ we have 
\begin{align}
		(\Xi^m_{g,n})_d = 0.
\end{align}
\end{conjecture}
The master relation is equivalent to Conjecture~\ref{conj:A=B}, the precise statement of this equivalence and its proof are delayed to the next section, see Theorem~\ref{thm: equivalence Xi=0 and A=B}.
For our applications, that is, for the case of tautological CohFTs and F-CohFTs, in particular for all semi-simple CohFTs, a weaker statement is sufficient:

\begin{theorem}[master relation in the Gorenstein quotient]\label{thm:master-relation} For any $g\geq 0$, $m,n\geq 1$, $2g-2+n+m>0$, and for any $\alpha \in R^*(\oM_{g,n+m})$ we have 
	\begin{align}
		\int_{\oM_{g,n+m}} \alpha \Xi^m_{g,n} \in Q[u].
	\end{align}
	In other words, for any $\alpha \in R^*(\oM_{g,n+m})$ we have  
	\begin{align}
		\deg \int_{\oM_{g,n+m}}  \alpha \Xi^m_{g,n}|_{u=1}  \leq 2g-2+m. 
	\end{align}
\end{theorem}
The latter statement can also be rephrased as the vanishing of $\int_{\oM_{g,n+m}}\alpha(\Xi^m_{g,n})_d $ for $d\geq 2g-1+m$ and any $\alpha\in R^*(\oM_{g,n+m})$, and we call this vanishing the master relation in the Gorenstein quotient. 

The rest of this section is devoted to a reduction of a  proof of this theorem to the following statement:
\begin{lemma} \label{lem:sufficient}
	Theorem~\ref{thm:master-relation} holds for any $\alpha$ if it holds for the classes $\alpha=\prod_{i=1}^m \psi_{n+i}^{b_i}$ for any $b_1,\dots,b_m\in\ZZ_{\geq 0}$. 
\end{lemma} 

\subsection{Intersection with a divisor} 
\label{sec:intersection-with-divisor}

For practical computations it is convenient to extend the definition above to $n=0$, where $\Xi^{m}_{g,0}$ is defined for $2g-2+m>0$ and is set to zero.

Let $\rho_1\colon \oM_{g-1,n+m+2} \to \oM_{g,n+m}$ be the boundary map.

\begin{lemma} We have
	\begin{align}
		\rho_1^* (\Xi^m_{g,n})_d = (\Xi^{m+2}_{g-1,n})_d.
	\end{align}
\end{lemma}
\begin{remark}
Note that if $d\geq 2g-1+m$, then $d\geq 2(g-1)-1+(m+2)$, so the degree remains in the range that should vanish in the Gorenstein quotient once it was there before the pull-back. 	
\end{remark}

The other type of the boundary maps $\rho \colon \oM_{g_1,k_1+1}\times \oM_{g_2,k_2+1} \to \oM_{g_1+g_2,n+m}$ is indexed by the parameters $g_1,g_2\geq 0$ such that $g_1+g_2=g$ and the sets of labels $I_1\sqcup I_2 = \{1,\dots,n+m\}$ such that $|I_i|=k_i$ and $2g_i-1+k_i>0$ for $i=1,2$. We always assume that the marked point labeled by $k_i+1$ on $\oM_{g_i,k_i+1}$ is the node, and the other $k_i$ points are ordered in such a way that the one-to-one map $\{1,\dots,k_i\}\to I_i\subset \{1,\dots,n+m\}$ preserves the order.

Let $\rho_2\colon \oM_{g_1,k_1+1}\times \oM_{g_2,k_2+1} \to \oM_{g_1+g_2,n+m}$ be the boundary map such that $m_1=|I_1\cap \{n+1,\dots,m\}|$ and $m_2=|I_2\cap \{n+1,\dots,m\}|$ are both non-zero for the corresponding sets of labels $I_1$ and $I_2$. Let $n_1= |I_1\cap \{1,\dots,n\}|$ and $n_2=|I_2\cap \{1,\dots,n\}|$.

\begin{lemma} We have
	\begin{align}
		\rho_2^* (\Xi^m_{g,n})_d = \sum_{d_1+d_2=d} (\Xi^{m_1+1}_{g_1,n_1})_{d_1}\otimes (\Xi^{m_2+1}_{g_2,n_2})_{d_2}. 
	\end{align}
In this formula we assume that the arguments of $\Xi$-classes correspond to the marked points on the corresponding components. 
\end{lemma}
\begin{remark}
	Note that if $d\geq 2g-1+m$, then either $d_1\geq 2g_1-1+(m_1+1)$ or $d_2\geq 2g_2-1+(m_2+1)$, so the degree on at least one of the two components is in the range that should vanish in the Gorenstein quotient once the degree of the original class was in the vanishing range before the pull-back. 	
\end{remark}

Let $\rho_3\colon \oM_{g_1,k_1+1}\times \oM_{g_2,k_2+1} \to \oM_{g,n+m}$ be the boundary map and assume that all points with the labels ${n+1},\dots,{n+m}$ are on the first component,  that is, $I_1\supset \{n+1,\dots,m\}$ and $I_1\cap \{n+1,\dots,m\}=\emptyset$. As before, let $n_1=|I_1\cap \{1,\dots,n\}|$ and $n_2=|I_2|\subset \{1,\dots,n\}$.
\begin{lemma} We have
	\begin{align}
		\rho_3^* (\Xi^m_{g,n})_d & = \sum_{d_1+d_2=d} (\Xi^{m+1}_{g_1,n_1})_{d_1}\otimes (\Xi^{1}_{g_2,n_2})_{d_2} \\ \notag & \quad + \sigma_*(\Xi^{m}_{g_1,n_1+1})_{d-2g_2}(\{a_j\}_{j\not\in I_2},a_{I_2})\otimes \lambda_{g_2}\DR_{g_2}\Big(\{a_i\}_{i\in I_2},-a_{I_2}\Big).
	\end{align}
In this formula we assume that the arguments of $\Xi$-classes correspond to the marked points on the corresponding components. In the last summand we use  $a_{I_2}\coloneqq \sum\nolimits_{i\in I_2} a_i$. The map $\sigma_*$ is induced by the relabeling of the marked points 
\begin{align}
(1,\dots,n_1+m+1)\mapsto (1,\dots,n_1,n_1+m+1,n_1+1,\dots,n_1+m)    
\end{align}
on $\oM_{g_1,k_1+1}$.
\end{lemma}
\begin{remark}
	Note that if $d\geq 2g-1+m$, then in the first summand either $d_1\geq 2g_1-1+(m+1)$ or $d_2\geq 2g_2-1+1$, so the degree on at least one of the two components is in the range that should vanish in the Gorenstein quotient once the degree of the original class was in the vanishing range before the pull-back. In the second summand, $d-2g_2 \geq 2g_1-1+m$ once $d \geq 2g-1+m$.
\end{remark}

The proof of all three lemmas is a straightforward application of the instructions on intersection of tautological classes given in~\cite{GraPan} and a formula for the intersection of a double ramification cycle with a divisor in~\cite{BSSZ}.

\subsection{Intersection with a kappa class} Lemmata in Section~\ref{sec:intersection-with-divisor} reduce the vanishing of $\int_{\oM_{g,n+m}} \alpha (\Xi^m_{g,n})_d$ for $d\geq 2g-1+m$ for an arbitrary tautological class $\alpha$ to the vanishing of the same intersection numbers for the classes $\kappa_{c_1,\dots,c_l}\prod_{i=1}^{n+m} \psi_i^{b_i}$. Indeed, any tautological class $\alpha$ is given as a sum of classes $(b_\Gamma)_*(\beta)$, where $\Gamma$ is a stable graph, $(b_\Gamma)_*$ is the corresponding boundary pushforward map, and $\beta$ is a tensor product of the products of $\psi$- and $\kappa$-classes assigned to the vertices of $\Gamma$. Using the projection formula we can rewrite $\int_{\oM_{g,n+m}} (b_\Gamma)_*(\beta) (\Xi^m_{g,n})_d$ as the product over the moduli spaces corresponding to the vertices of $\Gamma$ of the integrals of the classes obtained as components of $(b_\Gamma)^*(\Xi^m_{g,n})_d$, whose computation then follows the splitting formulas established in Lemmata in Section~\ref{sec:intersection-with-divisor}, multiplied for the products of $\psi$- and $\kappa$-classes. The latter Lemmata imply that $(b_\Gamma)^*(\Xi^m_{g,n})_d$ is an external tensor product of the classes $(\Xi^{m'}_{g',n'})_{d'}$ and $\lambda_g'\DR_{g'}$. Moreover, for at least one vertex we do obtain as the factor in the external tensor produc the class $(\Xi^{m'}_{g',n'})_{d'}$ with the condition $d'\geq 2g'-1+m'$. 
Thus, once we know that all integrals $\int_{\oM_{g',n'+m'}} \kappa_{c_1,\dots,c_l}\prod_{i=1}^{n'+m'} \psi_i^{b_i} (\Xi^{m'}_{g',n'})_d$ for $d'\geq 2g'-1+m'$ vanish, we obtain the vanishing of $\int_{\oM_{g,n+m}} \alpha (\Xi^m_{g,n})_d$ for $d\geq 2g-1+m$ for an arbitrary tautological class $\alpha$.
 
To this end, we have the following statement:

\begin{lemma} \label{lem:kapp-class} We have
	\begin{align} \label{eq:kapp-class}
		\int_{\oM_{g,n+m}} \kappa_{c_1,\dots,c_l}\prod_{i=1}^{n+m} \psi_i^{b_i} \Xi^m_{g,n}
		= \int_{\oM_{g,n+l+m}} \prod_{i=1}^{n} \psi_i^{b_i} \prod_{i=1}^{l} \psi_{n+i}^{c_i+1} \prod_{i=1}^m \psi_{n+l+i}^{b_{n+i}} \Xi^m_{g,n+l} |_{a_{n+1}=\cdots=a_{n+l}=0}.
	\end{align}
\end{lemma}

\begin{proof} Let $\pi_{n+1}\colon \oM_{g,n+1+m}\to\oM_{g,n+m}$ be the map that forgets the $(n+1)$-st marked point and shifts the labels of the last $m$ marked points by $-1$. Then $\pi_{n+1}^*\Xi^m_{g,n} = \Xi^m_{g,n+1}|_{a_{n+1}=0}$. In order to obtain~\eqref{eq:kapp-class} we iterate this formula $l$ times and use the projection formula for $\kappa_{c_1,\dots,c_l} = (\pi_{n+1}\cdots\pi_{n+l})_*\prod_{i=1}^l\psi_{n+i}^{c_i+1}$.
\end{proof}

\begin{remark} 
Equation~~\eqref{eq:kapp-class} preserves the degree, that is, it is bounded by $2g-2+m$ on the left hand side if and only if it is bounded by $2g-2+m$ on the right hand side. In particular, together with Lemmata of Section~\ref{sec:intersection-with-divisor}, this reduces Theorem~\ref{thm:master-relation} to a weaker statement, where the class $\alpha$ is a product of $\psi$-classes.
\end{remark}

\subsection{Intersection with a \texorpdfstring{$\psi$}{psi} class at a regular leg}

\begin{lemma} \label{lem:regular-psi} For any $i=1,\dots,n$, we have:
\begin{align} \label{eq:psi-regular}
	& {a_i\psi_i} (\Xi^m_{g,n})_d  = (\sigma_i)^*(\Xi^{m+1}_{g,n-1}(\{a_j\}_{j\not=i}))_{d+1} -  (\Xi^{m}_{g,n})_{d+1} 
	\\ \notag & 
	+\sum_{\substack{g_1+g_2=g\\ n_1+n_2=n \\ i\in I\subseteq \{1,\dots,n\} \\ |I|=n_2}} a_I (\rho_3)_*\left((\Xi^{m}_{g_1,n_1+1})_{d-2g_2}(\{a_j\}_{j\not\in I},a_I)\otimes \lambda_{g_2}\DR_{g_2}\Big(\{a_i\}_{i\in I},-a_I\Big) \right).
\end{align}
Here in the first term by $\sigma_i$ we mean the relabeling of the points
\begin{align}
	(1,\dots,n+m) \to (1,\dots,i-1,i+1,\dots,n+m,i).
\end{align} 
In the last sum we use $a_I\coloneqq \sum\nolimits_{j\in I} a_j$ and assume that $2g_1-1+n_1+m>0$ and $2g_2-1+n_2>0$, $\rho_3$ is the boundary map described above. 
\end{lemma}

\begin{proof} We just apply the following formula adapted from~\cite{BSSZ} to each term on the left hand side where the $i$-th marked point lies on a double ramification cycle, and then regroup the terms:
	\begin{align} \label{eq:psi-class-on-DR}
		& (a_i\psi_i+a\psi_{n+1}) \lambda_g \DR_g(a_1,\dots,a_n,-a) = 
		\\ \notag & \sum_{\substack{g_1+g_2=g\\ n_1+n_2=n\\ i\in I\subseteq \{1,\dots,n\} \\ |I|=n_2}} a_I \rho_*\left( \lambda_{g_1} \DR_{g_1}(\{a_j\}_{j\not\in I},-a,a_I)\otimes \lambda_{g_2} \DR_{g_2}(\{a_i\}_{i\in I},-a_I)  \right),
	\end{align} 
where $a = \sum_{i=1}^n a_i$, $a_I= \sum_{i\in I} a_i$, and $\rho\colon \oM_{g_1,n_1+2}\times \oM_{g_2,n_2+1}$ is the boundary map that glues the last marked points on both components into a node and maps the first $(n_1+1)$ points on the first component to the points with the indices $j\not\in I$ and $(n+1)$ and the first $n_2$ points on the second component to $i\in I$. 

In order to see how ~\eqref{eq:psi-regular} produces all terms on the right hand side of~\eqref{eq:psi-regular}, let us focus on a particular graph in $\PSSRT_{g,n,m}$ in the formula for $(\Xi^m_{g,n})_d$ and fix all degrees of $\psi$-classes associated via $\Psi(v_r)$ and $\bD(v)$, $v\in V(T)\setminus \{v_r\}$, to the half-edges in $E(T)$, except for one particular half-edge $e_1$ that connects $v_r$ to a vertex $v_1$. For simplicity assume both $v_r$ and $v_1$ are stable vertices, and the edge $e_1$ consists of two half-edges, $h_1$ attached to $v_r$ and $h_2$ attached to $v_1$, respectively. Since all degrees in the rest of the graph are fixed and $d$ is fixed, this edge is decorated by $a(e_1)^p\sum_{j=0}^p (-1)^j \psi(h_1)^j \psi(h_2)^{p-j}$. Assume further that $\sigma_i\in DL(e_1)$, that is, the leg $i$ is attached to $v_1$. 

Now, apply~\eqref{eq:psi-class-on-DR} to $a_i\psi_i\lambda_{g(v_1)}\DR_{g(v_1)}$. The term that corresponds to $-a_{n+1}\psi_{n+1}\lambda_g\DR_g$ in~\eqref{eq:psi-class-on-DR} turns the class $a(e_1)^p\sum_{j=0}^p \psi(h_1)^j \psi(h_2)^{p-j}$ associated to the edge $e_1$ into the sum of $-a(e_1)^{p+1}\sum_{j=0}^{p+1} (-1)^j \psi(h_1)^j \psi(h_2)^{p+1-j}$ and $a(e_1)^{p+1} \psi(h_1)^{p+1}$. The first term gives a summand in the formula for $-(\Xi^m_{g,n})_{d+1}$ on the right hand side of~\eqref{eq:psi-regular}, and the second term gives a summand in the second line of~\eqref{eq:psi-regular} for $g_2=g(v_1)$ and $DL(e_1)=\{\sigma_k\}_{k\in I}$. 

The cases when either $v_r$ or $v_1$ is unstable are analyzed in a similar way. In particular, is $v_1$ is unstable, then $DL(e_1)=\{e_i\}$, and we shouldn't use~\eqref{eq:psi-regular}, but rather contract the unstable vertex $v_1$, replace $h_1$ with the leg $\sigma_i$, and then note that multiplication by $a_i\psi_i$ increases the degree of the $\psi$-class assigned to $\sigma_i$ by $1$. This gives once again a contribution to the summand $-(\Xi^m_{g,n})_{d+1}$ on the right hand side of~\eqref{eq:psi-regular}. Note that this way we never obtain the degree of the $\psi$-class assigned to $\sigma_i$ equal to zero, which must be possible in the formula for $-(\Xi^m_{g,n})_{d+1}$, thus we need a counter term for this case hidden in the summand $(\sigma_i)^*(\Xi^{m+1}_{g,n-1}(\{a_j\}_{j\not=i}))_{d+1}$ in~\eqref{eq:psi-regular}.
\end{proof}

\begin{remark}\label{rem:regular-psi} Note that once on the left hand side of~\eqref{eq:psi-regular} $d\geq 2g-1+m$, then in all three summands on the right hand side of Equation~\eqref{eq:psi-regular} we have $d+1\geq 2g-1+(m+1)$, $d+1\geq 2g-1+m$, and $d-2g_2\geq 2g_1-1+m$ in the first, second, and third summand respectively. 
\end{remark}

Lemma~\ref{lem:regular-psi} together with Remark~\ref{rem:regular-psi} imply that the desired vanishing in the intersection numbers on the right hand side of Equation~\ref{eq:kapp-class} follows from the vanishing of
\begin{align}
	\deg \int_{\oM_{g,n+m}}  \prod_{i=1}^m \psi_{n+i}^{b_i} \Xi^m_{g,n}|_{u=1}     
\end{align}
in degrees $\geq 2g-1+m$ for all $g\geq 0$, $n,m\geq 1$ and all $b_1,\dots,b_m\in\ZZ_{\geq 0}$.
Together with Lemma~\ref{lem:kapp-class} and Lemmata of Section~\ref{sec:intersection-with-divisor}, this completes the proof of Lemma~\ref{lem:sufficient}.

\section{Proof of the main theorem} \label{sec:proof}

\subsection{Strategy of proof of Theorem~\ref{thm:A=B}} We first give a short outline of the proof and then we expand it to make it more precise. 

\subsubsection{Step 1} Using a combinatorial argument of~\cite{BLS-Omega} adapted to our case, we state a lemma which controls the difference between either: the $B$-class and the $\Xi$-class when $m\geq 2$, or the $A-B$-class and the $\Xi$-class when $m= 1$. 

\subsubsection{Step 2}A first consequence of this lemma is that the main theorem (Theorem~\ref{thm:main}) is equivalent to the master relation in the Gorenstein quotient (Theorem~\ref{thm:master-relation}).

\subsubsection{Step 3} We prove the master relation in the Gorenstein quotient. By a sequence of explicit intersections with tautological classes, we reduced Theorem~\ref{thm:master-relation} for any $\alpha$ to $\alpha=\prod_{i=1}^m \psi_{n+i}^{b_i}$, see Lemma~\ref{lem:sufficient}. To prove this special case, we use that the main theorem is already established in some special cases, see Lemma~\ref{lem:TrivialCohFT}, together with a second use of the lemma of step $1$.

\subsection{Step {1}: the key combinatorial lemma} The following lemma follows from a quite general argument described in a parallel paper, see ~\cite[Remark 5.4]{BLS-Omega}.

\begin{lemma}[Corollary of~\cite{BLS-Omega}] 
\label{lem:KeyLemma} The following two statements hold: 
\begin{enumerate}
	\item Let $m\geq 2$, $g\geq 0$, $n\geq 1$. Then, for any $d\geq 2g-1+m$, 
		\begin{align}
		(B^m_{g,n})_d -(-1)^d(\Xi^m_{g,n})_d
		\end{align}	
		is represented as a linear combination of tautological classes supported on graphs in $\SRT_{g,n,m}$ with at least one edge, where in each summand either one non-root vertex $v$ is decorated by $(\Xi^1_{g(v),|H_+(v)|})_{d(v)}$ with $d(v)\geq 2g(v)$ or the root vertex is decorated by $(B^m_{g(v_r),|H_+(v_r)|})_{d(v_r)}$ with $d(v_r)\geq 2g(v_r)-1+m$.
	\item Let $m=1$, $g\geq 0$, $n\geq 1$, $2g-1+n>0$. Then, for any $d\geq 2g$, 
	\begin{align}
		(B^1_{g,n}-A^{1}_{g,n})_d -(-1)^d(\Xi^1_{g,n})_d
	\end{align}
	is represented as a linear combination of tautological classes supported on graphs in $\SRT_{g,n,1}$ with at least one edge, where in each summand either one non-root vertex $v$ is decorated by $(\Xi^1_{g(v),|H_+(v)|})_{d(v)}$ with $d(v)\geq 2g(v)$ or the root vertex is decorated by $(B^1_{g(v_r),|H_+(v_r)|}-A^{1}_{g(v_r),|H_+(v_r)|})_{d(v_r)}$ with $d(v_r)\geq 2g(v_r)$. 
\end{enumerate}
\end{lemma}

\begin{proof} 
Instead of giving a proof of this lemma, we rather make a precise connection to the notation and statements in~\cite{BLS-Omega} thus expanding~\cite[Remark 5.4]{BLS-Omega}.

The class $\ensuremath{\Psi(v)}$ in our setting, which appears in the definitions of $\ensuremath{B_{g,n}^{m}}$ and $\ensuremath{\Xi_{g,n}^{m}}$, is replaced in \cite{BLS-Omega} by a different class denoted $\ensuremath{\mathbb{\Pi}(v)}$. We show that the replacement
\[
\ensuremath{\Psi(v)}\leftrightarrow\ensuremath{\mathbb{\Pi}(v)}
\]
induces precisely to the replacements
\begin{equation}
\ensuremath{B_{g,n}^{m}}\leftrightarrow{\hhh\hspace{-7pt}\hhh}_{g,n}^{m}\label{eq:replacement-B}
\end{equation}
and
\begin{align}
\ensuremath{\Xi_{g,n}^{m}} & \leftrightarrow{\Upsilon\hspace{-7pt}\Upsilon}_{g,n}^{m}\label{eq:replacement-Xi}
\end{align}
in \cite{BLS-Omega}.

Regarding (\ref{eq:replacement-B}), the class $B^m_{g,n}$ is defined in terms of $\Psi(v)$ in exactly the same way as the class ${{\hhh \hspace{-7pt} \hhh}}^m_{g,n}$ in terms of $\mathbb{\Pi}(v)$ in  \cite{BLS-Omega}, cf.~Equation~\eqref{eq:Deinition-of-B} and~\cite[Definition 3.2]{BLS-Omega} --- the conditions for the trees in $\SRT_{g,n,m}$ used in \emph{op.~cit.} are incorporated in the definition of $\LDLSRT_{g,n,m}$ in Section~\ref{sec:LeveledSRT}. 


The replacement (\ref{eq:replacement-Xi}) requires some further explanation. Indeed, the class $\Xi^m_{g,n}$ is defined in Equation~\eqref{eq:DefinitionOfXi} as a sum over pre-stable star rooted trees, while the class ${\Upsilon\hspace{-7pt}\Upsilon}^m_{g,n}$ is defined in~\cite[Equation~(4.3)]{BLS-Omega} as the sum over stable star rooted trees with at least two vertices and two extra summands, $\delta_{m,1}\bD_{g,n+1}$ and $\mathbb{\Pi}^{m}_{g,n}$. In fact, for $2g-2+m+n$ there are three possible cases of pre-stable star rooted trees:
\begin{enumerate}
    \item The root is stable and all other vertices are unstable. Under the push-forward in the definition of $\Xi^m_{g,n}$ this gives the class $\Psi(v_r)$ for the stable rooted tree that consists of just one vertex $v_r$. This matches the exceptional summand $\mathbb{\Pi}^{m}_{g,n}$ in the notation of~\cite{BLS-Omega} under the replacement $\Psi(v)\leftrightarrow \mathbb{\Pi}(v)$.
    \item The root is not stable. In this case $m=1$ and the star structure implies that there is just one more vertex in the corresponding pre-stable star rooted tree, decorated by $\bD_{g,n+1}$. This exactly matches the other exceptional summand in~\cite[Equation~(4.3)]{BLS-Omega}.
    \item The root is stable and there is at least one more stable vertex. In this case we can stabilize the pre-stable star rooted tree forgetting the unstable vertices, and we obtain a one-to-one correspondence with the stable rooted trees of level exactly $1$ in the notation of~\cite[Equation~(4.3)]{BLS-Omega}.
\end{enumerate}

Now the reader is referred to the statement of~\cite[Equation~(5.3)]{BLS-Omega}. For each $(g,n,m)$ it defines a combinatorial combination of particular stable rooted trees decorated by classes ${{\hhh \hspace{-7pt} \hhh}}^{m}_{g',n'}$, ${\Upsilon\hspace{-7pt}\Upsilon}^{m'}_{g',n'}$, and $A^1_{g',n'}$, $g'\leq g$, $n'\leq n$, $2g'+n'\leq 2g+n$, $m'=1,m$, that is equal to zero for purely combinatorial reasons according to~\cite[Theorem 5.1]{BLS-Omega}. Under the replacement $\Psi(v)\leftrightarrow \mathbb{\Pi}(v)$ this becomes a combinatorial statement about the classes $B^{m}_{g',n'}$, $\Xi^{m'}_{g',n'}$, and $A^1_{g',n'}$. This identity controls the difference $$\left(B_{g,n}^{m}-\delta_{m,1}A_{g,n}^{1}\right)_{d}-\left(-1\right)^{d}\left(\Xi_{g,n}^{m}\right)_{d},$$ where $\left(B_{g,n}^{m}-\delta_{m,1}A_{g,n}^{1}\right)_{d}$ corresponds to the analogue of the first term on the right-hand side of~\cite[Equation~(5.3)]{BLS-Omega}, and $\left(-1\right)^{d}\left(\Xi_{g,n}^{m}\right)_{d}$ corresponds to the trivial graph contribution from the final term in~\cite[Equation~(5.3)]{BLS-Omega}. Since the full structure of this combinatorial identity is rather cumbersome and irrelevant to the proof in the present paper, we just summarize its key properties that are stated in the lemma. 
\end{proof}

\begin{remark} In all cases, the structure of the tautological classes supported on the trees $T$ in $\SRT_{g,n,m}$ used in Lemma~\ref{lem:KeyLemma} is similar to the one used in the definition of $B^m_{g,n}$: each vertex is decorated by a class in  $R^*(\oM_{g(v),|H(v)|})\otimes_{\QQ}Q$, whose component of cohomological degree $d$ is a homogeneous polynomial in $a_i$'s of degree $d$ (this can be $\Xi^1_{g(v),|H_+(v)|}$, $\Psi(v)$, or $\bD(v)$, or eventually $(B^m_{g(v_r),|H_+(v_r)|}-\delta_{m,1}A^1_{g(v_r),|H_+(v_r)|})$ for the root vertex), and we multiply the push-forward $(b_T)_*$ of the tensor product of such classes by $\prod_{e\in E(T)} a(e)$. 
\end{remark}

\begin{remark}
	As we see from the statement of Lemma~\ref{lem:KeyLemma}, is it more convenient to use the class $\tilde\Xi^m_{g,n}\coloneqq \sum_{d\geq 0}(-1)^d(\Xi^m_{g,n})_d$.
\end{remark} 

\begin{remark}
Since the vertices of the graphs in $\SRT_{g,n,m}$ with at least one edge have strictly less negative Euler characteristic than $2g-2+m+n$, Lemma~\ref{lem:KeyLemma} allows to perform various inductive arguments. In particular, this lemma has two immediate corollaries, both proved by induction on $(g,n)$ with fixed $m$.	
\end{remark}

\subsection{Step {2}: equivalence between the master relation and DR/DZ}
It follows by induction using the key lemma that the master relation (Conjecture~\ref{conj:master-relation}) is equivalent to the generalized $A=B$ relations (Conjecture~\ref{conj:A=B}), more precisely:
\begin{theorem}\label{thm: equivalence Xi=0 and A=B}
	The following equivalences hold:
	\begin{enumerate}
		\item For any $(g,n)$ such that $g\geq 0$, $n\geq 1$, $2g-1+n>0$
\begin{multline*}
			\begin{cases}
				\deg \Xi_{g,n}^1\leq2g-1,
				\\
				\deg \Xi_{g',n'}^1\leq2g'-1,\ \text{for all}\ g'\leq g,\, n'\leq n,\, 2g'+n'< 2g+n
			\end{cases}
			 \qquad \Longleftrightarrow \qquad \\
			 \begin{cases}
					\deg (B^1_{g,n} - A^1_{g,n})\leq 2g-1, \\
					\deg (B^1_{g',n'} - A^1_{g',n'})\leq 2g'-1,\ \text{for all}\  g'\leq g,\, n'\leq n,\, 2g'+n'< 2g+n.
			 \end{cases}
\end{multline*}
		\item Fix $m\geq 2$. For any $g\geq 0$, $n\geq 1$, $2g-2+n+m>0$,
\begin{multline*}
	\begin{cases}
			\deg \Xi_{g,n}^m\leq 2g-2+m, \\
			\deg \Xi_{g',n'}^m\leq 2g'-2+m,\ \text{for all}\  g'\leq g,\, n'\leq n,\, 2g'+n'< 2g+n\\
			\deg \Xi_{g',n'}^1\leq 2g'-1,\ \text{for all}\ g'\leq g,\, n'\leq n,\, 2g'+n'< 2g+n
	\end{cases}
		\quad \Longleftrightarrow \quad \\
		\begin{cases}
			\deg B^m_{g,n} \leq 2g-2+m,\\ 
			\deg B^m_{g',n'} \leq 2g'-2+m,\ \text{for all}\  g'\leq g,\, n'\leq n,\, 2g'+n'< 2g+n\\
			\deg (B^1_{g',n'} - A^1_{g',n'})\leq 2g'-1,\ \text{for all}\  g'\leq g,\, n'\leq n,\, 2g'+n'< 2g+n.
\end{cases}
\end{multline*}
	\end{enumerate}
\end{theorem}

In particular, this equivalence holds in the Gorenstein quotient. Thus it suffices to prove the master relation in the Gorenstein quotient, that is Theorem~\ref{thm:master-relation}, to prove the main theorem. 

\subsection{Step 3: proof of the master relation in the Gorenstein quotient}

\label{sec:proof-master-Gor}

We prove Theorem~\ref{thm:master-relation} for $\alpha=\prod_{i=1}^m \psi_{n+i}^{b_i}$ which implies Theorem~\ref{thm:master-relation} for any $\alpha$ by Lemma~\ref{lem:sufficient}.

First, using the key lemma we deduce by induction the following statement:

\begin{corollary}\label{cor:monomials-psi} The following statements hold
\begin{enumerate}	
\item For any $g\geq 0$, $n\geq 1$, $2g-1+n>0$, and $b\geq 0$,
	\begin{align}
		\deg \int_{\oM_{g,n+1}} \psi_{n+1}^b \left(B^1_{g,n}-A^1_{g,n}-\tilde\Xi^1_{g,n}\right) \leq 2g-1
	\end{align}
	if 	
	\begin{align}
		\deg \int_{\oM_{g',n'+1}} \psi_{n+1}^b \left(B^1_{g',n'}-A^1_{g',n'}\right) \leq 2g'-1
	\end{align}
	holds for any $g'\geq 0$, $n'\geq 1$, such that $g'\leq g$, and $2g'+n'<2g+n$, and also
	\begin{align}
		\deg \int_{\oM_{g',n'+1}}  \tilde\Xi^{1}_{g',n'} \leq 2g'-1
	\end{align}
	holds for any $g'\geq 0$, $n'\geq 1$, such that $g'\leq g$, and $2g'+n'<2g+n$. 
\item Fix $m\geq 2$. For any $g\geq 0$, $n\geq 1$, and $b_1,\dots,b_m\in \ZZ_{\geq 0}$,
	\begin{align}
	\deg \int_{\oM_{g,n+m}} \left(\prod_{i=1}^m \psi_{n+i}^{b_i}\right) \left(B^m_{g,n}-\tilde\Xi^m_{g,n}\right) \leq 2g-2+m
\end{align}
if 
	\begin{align}
	\deg \int_{\oM_{g',n'+m}} \left(\prod_{i=1}^m \psi_{n'+i}^{b_i}\right) B^m_{g',n'} \leq 2g'-2+m
\end{align}
holds for any $g'\geq 0$, $n'\geq 1$, such that $g'\leq g$, and $2g'+n'<2g+n$, and also  
	\begin{align}
	\deg \int_{\oM_{g',n'+1}}  \tilde\Xi^{1}_{g',n'} \leq 2g'-1
\end{align}
holds for any $g'\geq 0$, $n'\geq 1$, such that $g'\leq g$, and $2g'+n'<2g+n$. 
\end{enumerate}
We recall our convention that a polynomial with negative degree is the zero polynomial.
\end{corollary}

Then, we combine the statements of Lemma~\ref{lem:TrivialCohFT} and Corollary~\ref{cor:monomials-psi}. Inductively, we obtain the following statement: 

\begin{lemma} \label{lem:Xi-monomials-psi} We have 
	\begin{align}
		\deg \int_{\oM_{g,n+m}} \left(\prod_{i=1}^m \psi_{n+i}^{b_i}\right) \tilde\Xi^m_{g,n} \leq 2g-2+m
	\end{align}
	for $m\geq 2$ and any $b_1,\dots,b_m\in \ZZ_{\geq 0}$ and 
	\begin{align}
		\deg \int_{\oM_{g,n+1}} \tilde\Xi^1_{g,n} \leq 2g-1.
	\end{align}
\end{lemma}

Note that we still don't have the statement for the degree of $\int_{\oM_{g,n+1}} \psi_{n+1}^b \tilde\Xi^1_{g,n}$  for any $b\in \ZZ_{\geq 0}$. However, by the pushforward formula which implies the string equation, we have 
\begin{align}
	\int_{\oM_{g,n+2}} \psi_{n+1}^b\psi_{n+2}^0 \tilde\Xi^2_{g,n} = \int_{\oM_{g,n+1}} \psi_{n+1}^{b-1} \tilde\Xi^1_{g,n} + \left(\sum_{i=1}^n a_i\right)\int_{\oM_{g,n+1}} \psi_{n+1}^{b} \tilde\Xi^1_{g,n}.
\end{align}
Hence, by induction on $b$, we obtain the following corollary of Lemma~\ref{lem:Xi-monomials-psi}:
\begin{corollary} \label{cor:Xi-monomials-psi-last-step} We have  
	\begin{align}
		\deg \int_{\oM_{g,n+1}} \psi_{n+1}^b \tilde \Xi^1_{g,n} \leq 2g-1
	\end{align}
	for any $b\in \ZZ_{\geq 0}$. 
\end{corollary}
 
 This completes the proof of Theorem~\ref{thm:master-relation} by Lemma~\ref{lem:sufficient}, and hence of Theorem~\ref{thm:A=B}.


\begin{thebibliography}{00}
	
\bibitem{ABLR21} A. Arsie, A. Buryak, P. Lorenzoni, P. Rossi. {Flat F-manifolds, F-CohFTs, and integrable hierarchies}. Comm. Math. Phys. {388} (2021), 291-328. 

\bibitem{AHIS} A.~Alexandrov, F.~Hern{\'a}ndez Iglesias, S.~Shadrin. Buryak-Okounkov formula for the n-point function and a new proof of the Witten conjecture. Int. Math. Res. Not. 2021, no. 18, 14296-14315.

\bibitem{BLRS} X.~Blot, D.~Lewa{\'n}ski, P.~Rossi, S. Shadrin. Stable tree expressions with Omega-classes and Double Ramification cycles.	
J. Geom. Phys.
209 (2025), 
Paper No. 105391, 17 pp.

\bibitem{BLS-Omega} X.~Blot, D.~Lewa{\'n}ski, S. Shadrin. Rooted trees with level structures, $\Omega$-classes and double ramification cycles. arXiv:2406.06205. 

\bibitem{BSS-Master} X.~Blot, A.~Sauvaget, S. Shadrin. The master relation for polynomiality and equivalences of integrable systems. 
Bull. Lond. Math. Soc. 57 (2025), no. 2, 599--604.

\bibitem{Bur} A.~Buryak. Double ramification cycles and integrable hierarchies. Comm. Math. Phys. 336 (2015), no. 3, 1085-1107.

\bibitem{BDGR1} A.~Buryak, B.~Dubrovin, J.~Gu\'er\'e, P.~Rossi. Tau-structure for the double ramification hierarchies. Comm. Math. Phys. 363 (2018), no. 1, 191-260.

\bibitem{BDGR20} A. Buryak, B. Dubrovin, J. Gu\'er\'e, P. Rossi. {Integrable systems of double ramification type}. Int. Math. Res. Not. {2020} (2020), no. 24, 10381-10446. 


\bibitem{BGR19} A.~Buryak, J.~Gu\'er\'e, P.~Rossi. DR/DZ equivalence conjecture and tautological relations. Geom. Topol. 23 (2019), no. 7, 3537--3600.

\bibitem{BHS} A.~Buryak, F.~Hern\'andez Iglesias, S.~Shadrin. A conjectural formula for $\DR_g(a,-a)\lambda_g$. Épijournal Géom. Algébrique 6 (2022), Art. 8, 17 pp.

\bibitem{BPS1}
A.~Buryak, H.~Posthuma, S.~Shadrin. A polynomial bracket for the Dubrovin-Zhang hierarchies. J. Differential Geom. 92 (2012), no. 1, 153-185. 

\bibitem{BPS2} A.~Buryak, H.~Posthuma, S.~Shadrin. On deformations of quasi-Miura transformations and the Dubrovin-Zhang bracket. J. Geom. Phys. 62 (2012), no. 7, 1639-1651.

\bibitem{BR16-quantum} A.~Buryak, P.~Rossi. Double ramification cycles and quantum integrable systems. Lett. Math. Phys. 106 (2016), no. 3, 289-317.

\bibitem{BR16} A.~Buryak, P.~Rossi.  Recursion relations for double ramification hierarchies. Comm. Math. Phys. 342 (2016), no. 2, 533–568. 

\bibitem{BR21} A.~Buryak, P.~Rossi.  {Extended $r$-spin theory in all genera and the discrete KdV hierarchy}. Adv. Math. {386} (2021), 107794.

\bibitem{BRS21} A.~Buryak, P.~Rossi,  S.~Shadrin. {Towards a bihamiltonian structure for the double ramification hierarchy}. Lett. Math. Phys. {111} (2021), Art. 13. 

\bibitem{BS22} A.~Buryak, S.~Shadrin. Tautological relations and integrable systems.  
Épijournal Géom. Algébrique 8 (2024), Art. 12, 44 pp.

\bibitem{BSSZ} A.~Buryak, S.~Shadrin, L.~Spitz, D.~Zvonkine. Integrals of $\psi$-classes over double ramification cycles. Amer. J. Math. 137 (2015), no. 3, 699-737.

\bibitem{DZ} B.~Dubrovin, Y.~Zhang. Normal forms of hierarchies of integrable PDEs, Frobenius manifolds and Gromov-Witten invariants. 	arXiv:math/0108160

\bibitem{FSZ} C.~Faber, S.~Shadrin, D.~Zvonkine.
Tautological relations and the r-spin Witten conjecture. 
Ann. Sci. Éc. Norm. Supér. (4) 43 (2010), no. 4, 621-658.

\bibitem{GraPan} T.~Graber, R.~Pandharipande. Constructions of nontautological classes on moduli
spaces of curves, Michigan Math. J. 51 (2003), 93-109.

\bibitem{Gub22} D. Gubarevich. {A conjectural formula for $\lambda_g\DR_g(a,-a)$ is true in Gorenstein quotient}. arXiv:2204.05396.

\bibitem{kon} M.~Kontsevich. Intersection theory on the moduli space of curves and the matrix Airy function. Commun. Math. Phys. 147, no. 1 (1992), 1-23. 

\bibitem{KM94} M.~Kontsevich, Y.~Manin. Gromov-Witten classes, quantum cohomology, and enumerative geometry. Comm. Math. Phys. 164 (1994), no. 3, 525-562.

\bibitem{LiuWangZhang}
S.~Liu, Z.~Wang, Y.~Zhang.
Linearization of Virasoro symmetries associated with semisimple Frobenius manifolds. arXiv:2109.01846 

\bibitem{LiuWang} X.~Liu, C.~Wang. On A Tautological Relation Conjectured By Buryak-Shadrin. arXiv:2402.14504

\bibitem{PPZ} R. ~Pandharipande,  A. ~Pixton, D. ~Zvonkine, Dimitri. Relations on $\overline{\mathcal{M}}_{g,n}$ via 3-spin structures. 
J. Amer. Math. Soc.
28 (2015), no. 1, 279--309.

\bibitem{Tel} C.~Teleman.
The structure of 2D semi-simple field theories. 
Invent. Math. 188 (2012), no. 3, 525-588.

\bibitem{wit-1} E.~Witten. Two-Dimensional Gravity and Intersection Theory on Moduli Space. Surveys in Differential Geometry (Cambridge, MA, 1990), 243-310. Bethlehem, PA: Lehigh University, 1991.

\bibitem{wit-2} E.~Witten, Algebraic geometry associated with matrix models of two-dimensional gravity, in Topological methods in modern mathematics (Stony Brook, NY, 1991), Publish or Perish, 1993, 235-269.

\end{thebibliography}
\end{document}